\documentclass[a4paper, twoside]{scrartcl}

\usepackage{amsfonts,amssymb,amsmath,amsthm,amstext,amssymb,mathtools}
\usepackage{caption,subcaption,float,color,xcolor,graphicx,enumerate}
\usepackage{dsfont} 
\usepackage[normalem]{ulem} 
\usepackage[authoryear,sort,round]{natbib} 

\usepackage{tabularx,multirow}

\usepackage{hyperref}
\usepackage[noabbrev,capitalise,nosort,nameinlink]{cleveref}
\crefname{assumption}{Assumption}{Assumptions}

\usepackage{tikz-cd} 
\usepackage{adjustbox}
\usepackage{dashbox}

\definecolor{verylightgray}{rgb}{0.85,0.85,0.85}

\definecolor{boxred}{rgb}{0.7,0,0}

\usepackage{etex}
\usepackage[a4paper, top=88pt, bottom=88pt, left=72pt, right=72pt, headsep=16pt, footskip=28pt]{geometry}
\usepackage{hyperref}
\hypersetup{
	colorlinks,
	linkcolor=blue,
	citecolor=blue,
	urlcolor=blue,
}
\newcommand*{\arXiv}[1]{\bgroup\color{blue}\href{https://arxiv.org/abs/#1}{arXiv:#1}\egroup}
\newcommand*{\doi}[1]{\bgroup\color{blue}\href{https://doi.org/#1}{doi:#1}\egroup}
\newcommand*{\email}[1]{\bgroup\color{blue}\href{mailto:#1}{#1}\egroup}
\renewcommand*{\url}[1]{\bgroup\color{blue}\href{#1}{#1}\egroup}
\usepackage{enumitem, moreenum}
\setlist[enumerate]{nosep}
\setlist[itemize]{nosep}
\usepackage{mleftright} \mleftright
\renewcommand{\qedsymbol}{$\blacksquare$}
\renewenvironment{proof}[1][\proofname]{\noindent{\bfseries\sffamily #1.} }{\hfill\qedsymbol\medskip}
\usepackage[labelfont={sf,bf}]{caption}
\DeclareCaptionLabelSeparator{figlabelsep}{\,\,\,}
\usepackage{scrlayer-scrpage, xhfill}
\automark[section]{section}
\setkomafont{pageheadfoot}{\normalcolor\sffamily}
\setkomafont{pagenumber}{\normalfont\normalsize\sffamily}
\clearpairofpagestyles
\let\oldtitle\title
\renewcommand{\title}[1]{\oldtitle{#1}\newcommand{\theshorttitle}{#1}}
\newcommand{\shorttitle}[1]{\renewcommand{\theshorttitle}{#1}}
\let\oldauthor\author
\renewcommand{\author}[1]{\oldauthor{#1}\newcommand{\theshortauthor}{#1}}
\newcommand{\shortauthor}[1]{\renewcommand{\theshortauthor}{#1}}
\cohead{\xrfill[0.525ex]{0.6pt}~\theshorttitle~\xrfill[0.525ex]{0.6pt}}
\cehead{\xrfill[0.525ex]{0.6pt}~\theshortauthor~\xrfill[0.525ex]{0.6pt}}
\newcommand{\theabstract}[1]{\par\bgroup\noindent\textbf{\textsf{Abstract.}} #1\egroup}
\newcommand{\thekeywords}[1]{\par\smallskip\bgroup\noindent\textbf{\textsf{Keywords.}}\newcommand{\and}{ $\bullet$ } #1\egroup}
\newcommand{\themsc}[1]{\par\smallskip\bgroup\noindent\textbf{\textsf{2020 Mathematics Subject Classification.}}\newcommand{\and}{ $\bullet$ } #1\egroup}
\newcommand*{\affilref}[1]{\ref{affiliation#1}}
\newcommand*{\affiliation}[3]{
	\footnotetext[#1]{\label{affiliation#2} #3}
}

\numberwithin{equation}{section}
\numberwithin{figure}{section}
\numberwithin{table}{section}

\newtheorem{theorem}{\sffamily Theorem}[section]
\newtheorem{corollary}[theorem]{\sffamily Corollary}
\newtheorem{proposition}[theorem]{\sffamily Proposition}
\newtheorem{lemma}[theorem]{\sffamily Lemma}
\theoremstyle{definition}
\newtheorem{definition}[theorem]{\sffamily Definition}
\newtheorem{remark}[theorem]{\sffamily Remark}
\newtheorem{example}[theorem]{\sffamily Example}

\newtheorem{assumption}[theorem]{\sffamily Assumption}

\crefname{conjecture}{Conjecture}{Conjectures}
\crefname{question}{Question}{Questions}

\newcommand*{\defeq}{\coloneqq}
\newcommand*{\qefed}{\eqqcolon}
\newcommand*{\defterm}{\textbf}

\DeclareMathOperator{\range}{\mathrm{ran}}
\DeclareMathOperator{\supp}{\mathrm{supp}}
\DeclareMathOperator{\spn}{\mathrm{span}}

\DeclareMathOperator*{\argmin}{arg\,min}

\renewcommand{\emptyset}{\varnothing}
\renewcommand{\geq}{\geqslant}
\renewcommand{\leq}{\leqslant}

\newcommand*{\Gammalim}{\mathop{\ensuremath{\Gamma}\text{-}\mathrm{lim}}}

\newcommand*{\N}{\mathbb{N}}
\newcommand*{\Naturals}{\mathbb{N}}
\newcommand*{\one}{\mathds{1}}
\newcommand*{\quark}{\setbox0\hbox{$x$}\hbox to\wd0{\hss$\cdot$\hss}}
\newcommand*{\R}{\mathbb{R}}
\newcommand*{\Reals}{\mathbb{R}}
\newcommand*{\eReals}{\overline{\Reals}}

\newcommand*{\rd}{\mathrm{d}}
\newcommand*{\Torus}{\mathbb{T}}
\newcommand*{\ud}{\, \rd}

\newcommand*{\Ball}[2]{\bgroup \color{green} B_{#2}(#1) \egroup}
\newcommand*{\cBall}[2]{B_{#2}(#1)}
\newcommand*{\ccBall}[2]{\bar{B}_{#2}(#1)}

\newcommand*{\BallExtra}[3]{\bgroup \color{green} B_{#2}^{#3}(#1) \egroup}
\newcommand*{\cBallExtra}[3]{B_{#2}^{#3}(#1)}

\newcommand*{\oRatio}[1]{R_{#1}}
\newcommand*{\cRatio}[1]{\bar{R}_{#1}}

\newcommand*{\rv}[1]{\boldsymbol{#1}}

\newcommand*{\prob}[1]{\mathcal{P}(#1)}

\newcommand*{\Normal}{\mathcal{N}}

\newcommand{\Borel}[1]{\mathcal{B}(#1)}
\newcommand{\cylindrical}[1]{\mathcal{E}(#1)}

\newcommand{\absval}[1]{\lvert #1 \rvert}

\newcommand{\innerprod}[2]{\langle #1 , #2 \rangle}
\newcommand{\norm}[1]{\lVert #1 \rVert}
\newcommand{\round}[1]{\lfloor #1 \rfloor}
\newcommand{\set}[2]{\{ #1 \mid #2 \}}
\newcommand{\bigabsval}[1]{\bigl\vert #1 \bigr\vert}

\newcommand{\bignorm}[1]{\bigl\Vert #1 \bigr\Vert}

\newcommand{\Absval}[1]{\left\vert #1 \right\vert}

\newcommand{\Set}[2]{\left\{ #1 \,\middle\vert\, #2 \right\}}

\newcommand{\todo}[1]{\bgroup\color{red}TODO: #1\egroup}
\newcommand{\ba}[1]{\bgroup\color{olive}Birzhan: #1\egroup}
\newcommand{\tjs}[1]{\bgroup\color{cyan}Tim: #1\egroup}
\newcommand{\hcl}[1]{\bgroup\color{brown}Han: #1\egroup}
\newcommand{\ik}[1]{\bgroup\color{violet}Ilja: #1\egroup}

\newcommand{\change}[1]{\bgroup\color{olive}#1\egroup}

\usepackage{acro}
\DeclareAcronym{BIP}{short=BIP, long=Bayesian inverse problem}
\DeclareAcronym{DFG}{short=DFG, long=Deutsche Forschungs\-gemein\-schaft}
\DeclareAcronym{MAP}{short=MAP, long=maximum a posteriori}
\DeclareAcronym{OM}{short=OM, long=Onsager--Machlup}
\DeclareAcronym{SPSD}{short=SPSD, long=self-adjoint and positive semi-definite}

\renewcommand{\change}[1]{#1}
\usepackage{censor}

\title{\texorpdfstring{$\boldsymbol{\mathsf{\Gamma}}$}{Gamma}-convergence of Onsager--Machlup functionals}
\subtitle{Part I: With applications to maximum a posteriori estimation\\in Bayesian inverse problems}
\shorttitle{$\mathsf{\Gamma}$-convergence of Onsager--Machlup functionals: Part I}
\author{%
	Birzhan~Ayanbayev\textsuperscript{\affilref{Warwick}}%
	\and
	Ilja~Klebanov\textsuperscript{\affilref{FUB}}%
	\and
	Han Cheng Lie\textsuperscript{\affilref{Potsdam}}%
	\and
	T.~J.~Sullivan\textsuperscript{\affilref{Warwick}}%
}
\shortauthor{B.~Ayanbayev, I.~Klebanov, H.~C.~Lie, and T.~J.~Sullivan}

\begin{document}

\maketitle

\affiliation{1}{Warwick}{Mathematics Institute and School of Engineering, University of Warwick, Coventry, CV4 7AL, United Kingdom (\email{birzhan.ayanbayev@warwick.ac.uk}, \email{t.j.sullivan@warwick.ac.uk})}
\affiliation{2}{FUB}{Freie Universit{\"a}t Berlin, Arnimallee 6, 14195 Berlin, Germany (\email{klebanov@zedat.fu-berlin.de})}
\affiliation{3}{Potsdam}{Institut f\"ur Mathematik, Universit\"at Potsdam, Campus Golm, Haus 9, Karl-Liebknecht-Stra{\ss}e 24--25, Potsdam OT Golm 14476, Germany (\email{hanlie@uni-potsdam.de})}

\begin{abstract}
	\theabstract{The Bayesian solution to a statistical inverse problem can be summarised by a mode of the posterior distribution, i.e.\ a MAP estimator.
The MAP estimator essentially coincides with the (regularised) variational solution to the inverse problem, seen as minimisation of the Onsager--Machlup functional of the posterior measure.
An open problem in the stability analysis of inverse problems is to establish a relationship between the convergence properties of solutions obtained by the variational approach and by the Bayesian approach.
To address this problem, we propose a general convergence theory for modes that is based on the $\Gamma$-convergence of Onsager--Machlup functionals, and apply this theory to Bayesian inverse problems with Gaussian and edge-preserving Besov priors.
Part~II of this paper considers more general prior distributions.
}
	\thekeywords{Bayesian inverse problems%
\and%
$\Gamma$-convergence%
\and%
maximum a posteriori estimation%
\and%
Onsager--Machlup functional%
\and%
small ball probabilities%
\and%
transition path theory%
}
	\themsc{49Q20
\and
60B11
\and
49J45
\and
49K40
\and
62F15
}
\end{abstract}

\section{Introduction}

In diverse applications such as Bayesian inference and the transition path analysis of diffusion processes, it is important to be able to summarise a probability measure $\mu$ on a possibly infinite-dimensional space $X$ by a single distinguished point of $X$ --- a point of maximum probability under $\mu$ in some sense, i.e.\ a \emph{mode} of $\mu$.
If $\mu$ is an absolutely continuous measure on a finite-dimensional Euclidean space $X$, then the modes of $\mu$ are the maximisers of its Lebesgue density.
In the Bayesian statistical context, if $\mu$ is the posterior measure, then the modes of $\mu$ are precisely the \ac{MAP} estimators.
If $X$ is an infinite-dimensional Banach space $X$, then a Lebesgue density is not available.
In this case it has become common to define modes using the posterior probabilities of norm balls in the small-radius limit.
Under suitable conditions, such modes admit a variational characterisation as the minimisers of an appropriate \ac{OM} functional.
Heuristically, such an \ac{OM} functional plays the role of the negative logarithm of the ``Lebesgue density'' of $\mu$, but the rigorous formulation of this relationship requires some care.
In a statistical context, this variational characterisation of modes suggests a connection between the fully Bayesian approach and the regularised variational approach to inverse problems:
the negative logarithm of the prior acts as a regulariser for the misfit (i.e.\ for the negative log-likelihood).

A significant challenge to exploiting this connection is the lack of a suitable convergence theory.
This is because the stability properties of \ac{MAP} estimators are poorly understood.
In particular, it is not known under what circumstances mild perturbations of the setup of a \ac{BIP} lead to mild perturbations of the posterior distribution \emph{and} to mild perturbations of its \ac{MAP} estimators.
\change{Typical examples of perturbations include those arising from finite-dimensional truncation of an infinite-dimensional prior;
numerical approximation of an ideal forward operator within the likelihood, e.g.\ the solution operator of a differential equation;
perturbation of observed data;
or limiting procedures such as small-noise limits.}

In the last decade, beginning with the seminal work of \citet{Stuart2010}, many articles have studied the well-posedness and stability of \ac{BIP}s in function spaces.
However, in general, the stability of the posterior and of its \ac{MAP} estimators are ``orthogonal'' questions:
two posterior probability measures can be arbitrarily close in a strong sense such as Kullback--Leibler (relative entropy) distance and still have \ac{MAP} estimators that are at constant distance from one another;
conversely, even equality of \ac{MAP} estimators says nothing about the similarity of the full posteriors.
The situation vis-a-vis convergence is even less satisfying, as the following examples show:

\begin{example}
	\label{example:MAP_counterexamples}
	\begin{enumerate}[label=(\alph*)]
		\item \label{item:same_MAP_far_apart_KL}
		The normal distributions $\mu^{(1)} \defeq \Normal(0, 1)$ and $\mu^{(\sigma)} \defeq \Normal(0, \sigma^{2})$ on $\Reals$ have the same unique mode at $0$, but can be very far apart in the Kullback--Leibler sense:
		\begin{align}
			\label{eq:same_MAP_far_apart_KL}
			\textup{KL} \bigl( \mu^{(1)} \big\| \mu^{(\sigma)} \bigr) & = \frac{\sigma^{-1} - 1 + \log \sigma}{2} \to + \infty \quad \text{as $\sigma \to 0$ or $\sigma \to +\infty$.}
		\end{align}

		\item \label{item:close_KL_far_apart_MAP}
		In the other direction, fix a large $r \gg 1$ and consider for $t \in (-1, 1)$ the following Gaussian mixture distribution $\mu^{(t)}$ on $\Reals$ with Lebesgue density $\rho^{(t)} \colon \Reals \to [0, \infty)$ illustrated in \Cref{fig:MAP_counterexamples}(\subref{fig:close_KL_far_apart_MAP}) and given by
		\begin{align}
			\label{eq:close_KL_far_apart_MAP}
			\rho^{(t)} (x) \defeq \frac{(1 + t) \exp( - \frac{1}{2} (x - r)^{2} ) + (1 - t) \exp( - \frac{1}{2} (x + r)^{2} )}{2 \sqrt{2 \pi}} .
		\end{align}
		For large values of $r$, $\rho^{(t)}$ has two local maximisers near $r$ and $-r$.
		When $t > 0$, the local maximiser near $r$ is the unique mode;
		when $t < 0$, the local maximiser near $- r$ is the unique mode.
		However, $\mu^{(t)}$ and $\mu^{(-t)}$ are very close in the Kullback--Leibler sense:
		e.g.\ for $r = 5$, $\textup{KL} \bigl( \mu^{(t)} \big\| \mu^{(-t)} \bigr) \approx 10^{5 / 4} t^{9 / 4}$ as $t \to 0$.
		
		One might argue that at least the cluster points as $t \to 0$ of the modes of the measures $\mu^{(t)}$ yield the two modes at $x = \pm r$ of the symmetric Gaussian mixture $\mu^{(0)}$.
		However, even this situation cannot be expected to hold true in general, as the next example shows.
		
		\item \label{item:pointwise_convergent_density}
		For $n \in \Naturals$, let $\mu^{(n)}$ be the probability measure on $\Reals$ with Lebesgue density $\rho^{(n)} \colon \Reals \to [0, \infty)$ illustrated in \Cref{fig:MAP_counterexamples}(\subref{fig:pointwise_convergent_density}) and given by
		\begin{align}
			\label{eq:pointwise_convergent_density}
			\rho^{(n)} (x)
			\defeq
			\frac{ \exp \bigl( - \tfrac{1}{2} (x - 1)^{2} \bigr) + \one [ x \geq  0 ] 4 n^{2} x^{2} \exp ( - n^{2} x^{2} ) }{ \sqrt{2 \pi} + \sqrt{\pi} / n } ,
		\end{align}
		where $\one [P] \defeq 1$ if $P$ is true and $\one [P] \defeq 0$ if $P$ is false.
		The densities converge pointwise but not uniformly to the Gaussian distribution $\mu^{(\infty)} = \Normal(1, 1)$ with density $\rho^{(\infty)} (x) \propto \exp \bigl( - \tfrac{1}{2} (x - 1)^{2} \bigr)$, which has a unique mode at $x = 1$.
		Convergence in the Kullback--Leibler sense also holds, with $\textup{KL} \bigl( \mu^{(\infty)} \big\| \mu^{(n)} \bigr) \approx \tfrac{1}{n}$.
		However, each $\mu^{(n)}$ has a unique mode at approximately $x \approx \tfrac{1}{n}$ and the unique cluster point of this sequence of modes is $x = 0 \neq 1$.
		Thus, even in finite-dimensional settings, pointwise convergence of densities (and hence of \ac{OM} functionals) does not imply convergence of modes.
	\end{enumerate}
\end{example}

\begin{figure}
	\begin{subfigure}{0.5\textwidth}
		\centering
		\includegraphics[height=15em]{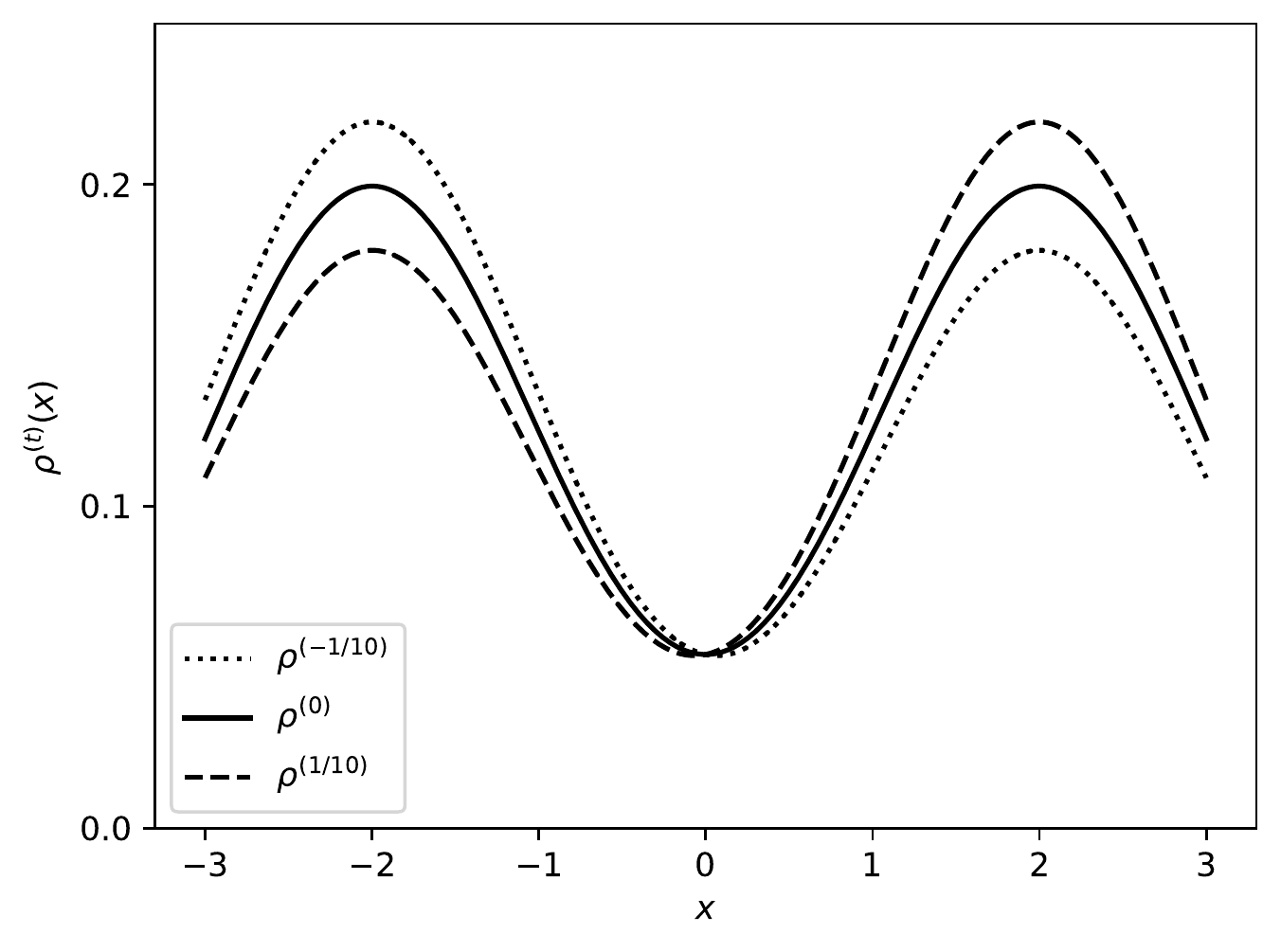}
		\caption{}
		\label{fig:close_KL_far_apart_MAP}
	\end{subfigure}
	\begin{subfigure}{0.5\textwidth}
		\centering
		\includegraphics[height=15em]{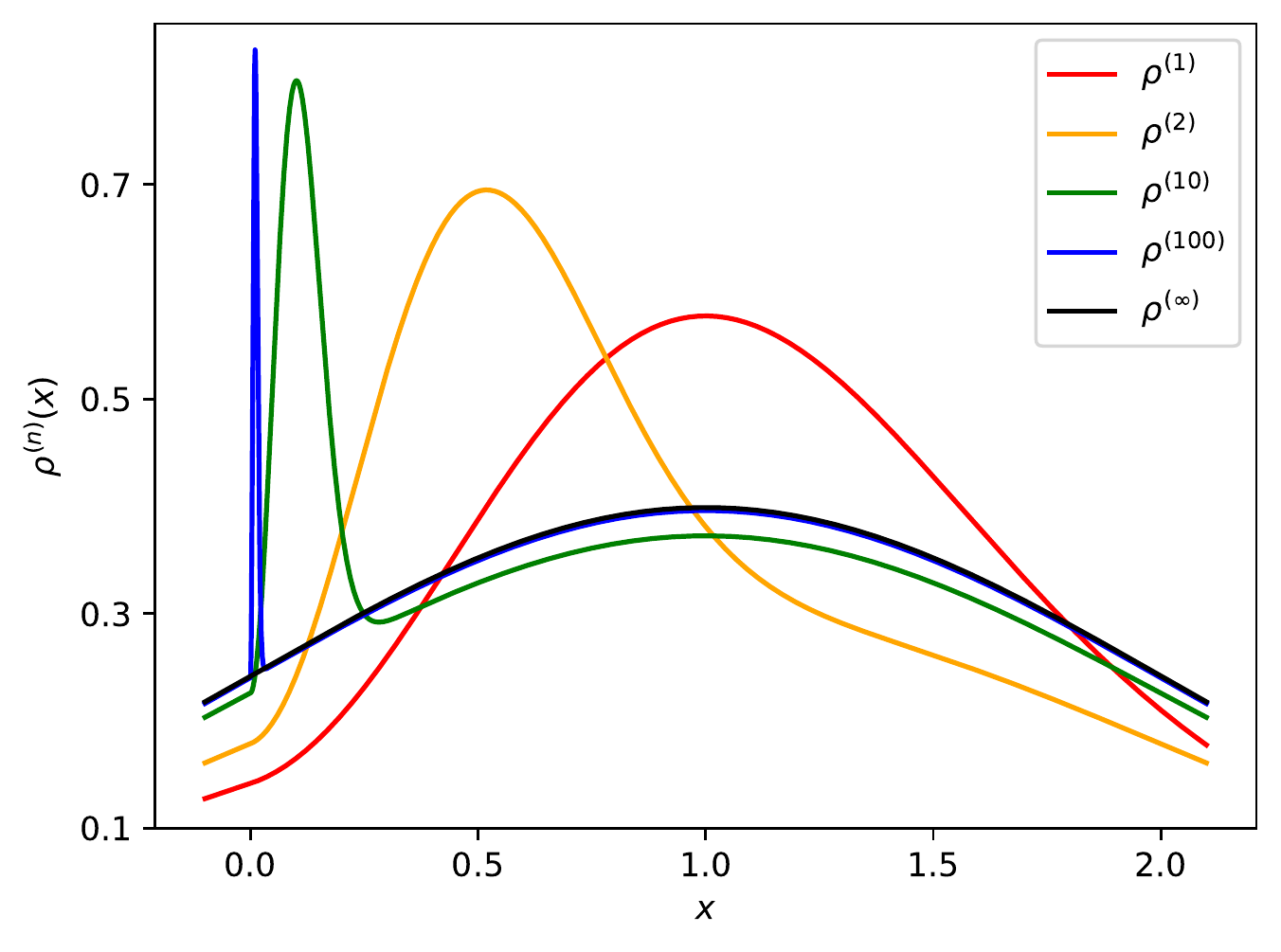}
		\caption{}
		\label{fig:pointwise_convergent_density}
	\end{subfigure}
	\caption{(a)~The densities $\rho^{(t)}$ from \eqref{eq:close_KL_far_apart_MAP} for $t \approx 0$ are close in the Kullback--Leibler sense but their modes are far apart, at $x \approx \pm r = \pm 2$ respectively.
	However, the cluster points of the modes of $\rho^{(t)}$ as $t \to 0$ yield the modes of $\rho^{(0)}$.\newline
	(b)~The densities $\rho^{(n)}$ from \eqref{eq:pointwise_convergent_density} for $n \in \{ 1, 2, 10, 100, \infty \}$.
	For $n \in \Naturals$, the unique mode of $\rho^{(n)}$ is at $x \approx \frac{1}{n}$, whereas the unique mode of $\rho^{(\infty)}$ is at $x = 1$, even though $\rho^{(n)} \to \rho^{(\infty)}$ pointwise.}
	\label{fig:MAP_counterexamples}
\end{figure}

Given the variational characterisation of modes as minimisers of an \ac{OM} functional, it seems natural to assess the convergence (and hence stability) of modes using a notion of convergence that is appropriate for variational problems, i.e.\ one for which the convergence of functionals implies the convergence of minimisers.
The notion of $\Gamma$-convergence, as introduced by De Giorgi and collaborators from the 1970s onwards \citep{DeGiorgi2006}, fulfils exactly this r\^ole and, in particular, overcomes the shortcomings of pointwise convergence as illustrated in \Cref{example:MAP_counterexamples}\ref{item:pointwise_convergent_density}.
Indeed, if the densities in \Cref{example:MAP_counterexamples}\ref{item:pointwise_convergent_density} were uniformly convergent, then they would be continuously convergent and hence $\Gamma$-convergent as well, and the pathological non-convergence of modes would have been avoided. Therefore, the strategy followed by this article consists in the following:
\begin{enumerate}
	\item Formulate the problem of finding modes of probability measures on potentially infinite-dimensional spaces as a variational problem for the associated \ac{OM} functionals; and
	\item Study the $\Gamma$-convergence properties of such problems, in order to obtain criteria for the convergence and stability of such modes.
\end{enumerate}

The remainder of this article is structured as follows.
\Cref{sec:related} gives an overview of related work in the theory of modes for measures on infinite-dimensional spaces and \Cref{sec:notation} sets out some notation and basic results for the rest of the article.
\Cref{sec:modes_and_OM_functionals} explores the correspondence between modes and minimisers of \ac{OM} functionals, and hence demonstrates that $\Gamma$-convergence of \ac{OM} functionals is the correct notion of convergence to ensure convergence of modes.
\Cref{sec:convergence_of_OM_functionals} develops this idea in two prototypical settings, namely Gaussian and Besov-1 measures, which are frequently used as Bayesian prior distributions;
these results can then be transferred to measures that are absolutely continuous with respect to these paradigmatic examples and can be interpreted as the corresponding posterior measures.
More general prior measures, which include Cauchy measures and Besov-$p$ measures for $1\leq p\leq 2$ are treated in Part~II of this paper \citep{AyanbayevKlebanovLieSullivan2021_II}.
In \Cref{sec:MAP_in_BIP}, these ideas are then applied to the convergence and stability of \ac{MAP} estimators for \ac{BIP}s, for which Gaussian and Besov-1 measures are prototypical prior distributions.
Some conclusions and suggestions for further work are given in \Cref{sec:closing}.
Standard definitions and results relating to $\Gamma$-convergence are collected in \Cref{sec:Gamma} and technical supporting results are given in \Cref{sec:technical}.

\section{Overview of related work}
\label{sec:related}

In stochastic analysis and mathematical physics, the interpretation of the minimisers of \acl{OM} functionals over path spaces as most probable paths appears to be due to \citet{DuerrBach1978}.
The \ac{OM} functionals of diffusion processes, and hence the determination of maximum a posteriori paths, have been considered by e.g.\ \citet{Zeitouni1989} and \citet{DemboZeitouni1991}.
It is important to note that simply determining the \ac{OM} functional on some $n$-dimensional approximation space and then taking a limit as $n \to \infty$ can fail to yield the correct \ac{OM} functional as defined in terms of ratios of small ball probabilities by \eqref{eq:Onsager--Machlup}.
This is because the space on which the \ac{OM} functional is finite can be ``smoother'' than the full space on which small ball probabilities are defined \citep{DashtiLawStuartVoss2013}.

Recent years have seen a growing interest in the well-posedness and stability of \ac{BIP}s in function spaces, a perspective proposed by a seminal article of \citet{Stuart2010} that has stimulated many follow-on works and generalisations (e.g.\ \citealp{DashtiHarrisStuart2012, Hosseini2017, Latz2020, Sprungk2020, Sullivan2017}).
There is also complementary theory of discretisation invariance, sometimes referred to as the ``Finnish school'', which in some sense treats the finite-dimensional discrete versions of \ac{BIP}s as the primary objects of interest but pays careful attention to their limiting properties as the discretisation dimension tends to infinity (e.g.\ \citealp{Lehtinen1989, LassasSiltanen2004, LassasSaksmanSiltanen2009, Lasanen2012a, Lasanen2012b}).
However, the robustness studies in these works have focussed on the robustness of the posterior measure in a distributional sense such as the Hellinger, Kullback--Leibler, or Wasserstein sense and, as \Cref{example:MAP_counterexamples} shows, these are insufficient to ensure robustness of modes or \ac{MAP} estimators.
\change{Our results can be seen as contributions to this field in the sense that they establish stability/convergence of \ac{MAP} estimators in a setting that is not limited to finite-dimensional or even linear spaces.}

In the \ac{BIP} context, the definition of a \ac{MAP} estimator as the centre of a norm ball that has maximum posterior probability in a small-radius limit appears to be due to \citet{DashtiLawStuartVoss2013}.
As \citet{DashtiLawStuartVoss2013} note, a similar definition of a \emph{maximal point} was given earlier by \citet{Hegland2007}, but that analysis was implicitly limited to the finite-dimensional setting, since it assumed finiteness of the Cameron--Martin norm.
The context of \citet{DashtiLawStuartVoss2013} was limited to a separable Hilbert space\footnote{\citet{DashtiLawStuartVoss2013} stated their results in separable Banach spaces.
However, as pointed out by \citet{Wacker2020}, their proof techniques are valid only in the case that $X$ is a Hilbert space and not in the Banach space case.} $X$ equipped with a Bayesian posterior measure $\mu$ that was absolutely continuous with respect to a centred non-degenerate Gaussian reference measure $\mu_{0}$.
In this setting, \citet{DashtiLawStuartVoss2013} established the existence of \ac{MAP} estimators and characterised them as the minimisers of the \ac{OM} functional, which they further identified as the sum of the log-likelihood and the \ac{OM} functional of $\mu_{0}$.
When read in the context of a general probability measure on a metric space, rather than the original setting of a Bayesian posterior on a Hilbert space, the definition of \citet{DashtiLawStuartVoss2013} is essentially the definition of a \emph{strong mode} (\Cref{defn:strong_mode}).

The work of \citet{DashtiLawStuartVoss2013} has been extended in multiple ways.
\citet{DunlopStuart2016} proved the connection between \ac{MAP} estimators and \ac{OM} functionals in the setting of piecewise continuous inversion, where the prior is defined in terms of a combination of Gaussian random fields.
Recently, \citet{Kretschmann2019} has corrected some technical deficiencies of \citet{DashtiLawStuartVoss2013}.

The definition of a (strong) \ac{MAP} estimator for $\mu$ given by \citet{DashtiLawStuartVoss2013} was relaxed to that of a \emph{weak \ac{MAP} estimator} by \citet{HelinBurger2015}, in which comparisons between the masses of balls are only performed for balls whose centres differ by an element of a topologically dense subspace $E$ of $X$.
\citet{HelinBurger2015} showed that this weak \ac{MAP} estimator has a close relationship with the zeroes of the logarithmic derivative $\beta_{h}^{\mu} \defeq \frac{\rd (d_{h} \mu)}{\rd \mu}$ of $\mu$, where
\[
	d_{h} \mu (A) \defeq \lim_{t \to 0} \frac{\mu (A + t h) - \mu (A)}{t}
	\quad
	\text{for measurable $A \subseteq X$}
\]
is the Fomin (directional) derivative of $\mu$ in the direction $h \in X$.
The initial analysis of the weak \ac{MAP} estimator relied upon the existence of a continuous representative for $\beta_{h}^{\mu}$, which could not be guaranteed for several important applications, notably the edge-preserving Besov prior with $p = 1$.
By focussing on the Radon--Nikodym derivative $r_{h}^{\mu} \defeq \frac{\rd \mu ( \quark - h )}{\rd \mu}$ instead of $\beta_{h}^{\mu}$, the analysis of \citet{AgapiouBurgerDashtiHelin2018} remedied this shortcoming, posed the definitions and results in more general terms of modes of probability measures rather than \ac{MAP} estimators of Bayesian posteriors, and also considered local (rather than global) strong and weak modes.
The equivalence of strong and weak modes when $E$ is dense in $X$ and under a \emph{uniformity condition} on $\mu$ was established by \citet{LieSullivan2018}, who worked in the more general context of measures on metrisable topological vector spaces.

Finally, we mention that the definition of a strong mode is unsuitable for probability measures with bounded support, and especially to such measures with essential discontinuities in the density.
Examples of such measures include uniform measures on bounded subsets of $X$. The definition is unsuitable because it excludes ``obvious'' modes on the boundary of the support.
The recent \emph{generalised mode} of \citet{Clason2019GeneralizedMI} addresses this deficiency.

\section{Preliminaries and notation}
\label{sec:notation}

\subsection{General notation and assumptions}

Throughout this article, $X$ will denote either a topological space, a metric space, a separable Banach or a Hilbert space.
When thought of as a measurable space, $X$ will be equipped with its Borel $\sigma$-algebra $\Borel{X}$, which is generated by the collection of all open sets.
When $X$ is a metric space, we write $\cBall{x}{r}$ for the open ball in $X$ of radius $r$ centered on $x$.
The set of all probability measures on $(X, \Borel{X})$ will be denoted $\prob{X}$;
we denote typical probability measures by $\mu$, $\mu_{0}$, $\mu^{(n)}$ for $n\in\Naturals \cup \{ \infty \}$ etc.
The topological support of $\mu\in\prob{X}$ defined on a metric space $X$ is
\begin{equation}
	\label{eq:support_of_measure}
	\supp (\mu) \defeq \set{ x \in X }{ \text{for all $r > 0$, } \mu ( \cBall{x}{r} ) > 0 } ,
\end{equation}
which is always a closed subset of $X$.

We write $\eReals$ for the extended real line $\Reals \cup \{ \pm \infty \}$, i.e.\ the two-point compactification of $\Reals$.

For $0 < p \leq \infty$, we write $\ell^{p} \defeq \ell^{p}(\Naturals)$ for the real sequence space of $p$\textsuperscript{th}-power summable sequences, and bounded sequences in the case $p=\infty$.
Furthermore, given $\gamma=(\gamma_{k})_{k \in \Naturals} \in \Reals_{> 0}^{\Naturals}$, we write
\begin{equation}
	\label{eq:gamma_weighted_sequence_space}
	\ell^{p}_{\gamma}
	\defeq
	\Set{ h \in \Reals^{\Naturals} }{ ( h_{k} / \gamma_{k} )_{k \in \Naturals} \in \ell^{p} },
	\qquad
	\norm{ h }_{\ell^{p}_{\gamma}}
	\defeq
	\bignorm{ ( h_{k} / \gamma_{k} )_{k \in \Naturals}}_{\ell^{p}}
\end{equation}
for the corresponding weighted $\ell^{p}$ space.
It is well known that $\norm{ \quark }_{\ell^{p}}$ (and hence $\norm{ \quark }_{\ell_{\gamma}^{p}}$) is a complete quasinorm when $p>0$, a Banach norm when $p \geq 1$, and a Hilbert norm when $p = 2$.

\subsection{Onsager--Machlup functionals}
\label{ssec:OM_functionals}

We recall here the definition of an \acl{OM} functional for a measure. The minimisers of \ac{OM} functionals will turn out to be the modes of the measure (see \Cref{sec:modes_and_OM_functionals}).
We also stress a property that is already implicitly used without a name in the modes literature, one that essentially ensures that the modes must lie in the domain of the \ac{OM} functional.

\begin{definition}
	\label{defn:Onsager--Machlup}
	Let $X$ be a metric space and let $\mu \in \prob{X}$.
	We say that $I = I_{\mu} = I_{\mu, E} \colon E \to \Reals$, with\footnote{\label{footnote:support_avoids_0/0}We insist on defining the real-valued version of $I$ only on a subset of $\supp (\mu)$ since, for $x_{2} \notin \supp (\mu)$, the ratio in the limit in \eqref{eq:Onsager--Machlup} is infinite or even undefined for small enough $r > 0$.} $\emptyset\neq E \subseteq \supp(\mu) \subseteq X$, is an \defterm{Onsager--Machlup functional} (\ac{OM} functional) for $\mu$ if
	\begin{equation}
		\label{eq:Onsager--Machlup}
		\lim_{r \searrow 0} \frac{ \mu ( \cBall{x_{1}}{r} ) }{ \mu ( \cBall{x_{2}}{r} ) }
		=
		\exp ( I(x_{2}) - I(x_{1}) )
		\text{ for all $x_{1}, x_{2} \in E$.}
	\end{equation}
	We say that \defterm{property $M(\mu, E)$} is satisfied if, for some $x^{\star} \in E$,
	\begin{equation}
		\label{eq:liminf_small_ball_prob}
		x \in X \setminus E \implies \lim_{r \searrow 0} \frac{ \mu ( \cBall{x}{r} ) }{ \mu ( \cBall{x^{\star}}{r} ) } = 0 ,
	\end{equation}
	and in this situation we extend $I$ to a function $I \colon X \to \eReals$ with $I(x) \defeq + \infty$ for $x \in X \setminus E$.
\end{definition}

\begin{remark}[The $M$-property]%
	\Cref{lem:vanishing_small_ball_prob} establishes some basic facts about property $M(\mu, E)$.
	In particular, \Cref{lem:vanishing_small_ball_prob}\ref{item:vanishing_small_ball_prob_xstar_or_xstarstar} shows that property $M (\mu, E)$ does not depend on the choice of reference point $x^{\star} \in E$, provided that $\mu$ has an \ac{OM} functional on $E$.

	Every measure $\mu$ admits an \ac{OM} functional if $E$ is taken to be small enough, e.g.\ a singleton subset of $\supp (\mu)$.
	Therefore, there is a natural desire to have $E$ be ``maximal'' in some sense.
	Property $M(\mu, E)$ means that the set $E$ is the ``maximal'' set on which the \ac{OM} functional assumes finite values.
	
	It is tempting but incorrect to read property $M(\mu, E)$ as saying that $\mu$ somehow concentrates upon $E$.
	A straightforward counterexample is given by any non-degenerate Gaussian measure $\mu$ with infinite-dimensional Cameron--Martin space $H(\mu)$, \change{such as the law $\mu$ of standard Brownian motion on $X = C([0, 1]; \Reals)$ with $H(\mu) = H^{1}([0, 1]; \Reals)$.
	In this situation,} property $M(\mu, H(\mu))$ holds \citep[Lemma~3.7]{DashtiLawStuartVoss2013} and yet $\mu(H(\mu)) = 0$ \citep[Theorem~2.4.7]{Bogachev1998Gaussian}.
	Rather, the purpose of property $M(\mu, E)$ is to ensure that the global weak modes (see \Cref{defn:global_weak_mode}) of $\mu$ lie in $E$ and are precisely the minimisers of its extended \ac{OM} functional (see \Cref{prop:weak_mode_OM_minimiser}).
	Furthermore, it is essentially the $\limsup$ part of the limit in \eqref{eq:liminf_small_ball_prob} that ensures this;
	\Cref{example:Minimiser_of_OM_not_weak_mode} shows that if we weaken \eqref{eq:liminf_small_ball_prob} by considering the limit inferior instead of the limit, then --- even for very simple choices of $E$ --- the desired correspondence may break down.
\end{remark}

\begin{remark}[Topological considerations]%
	Note that in defining the \ac{OM} functional here and various notions of mode / \ac{MAP} estimator later on, we use \emph{open} balls (following e.g.\ \citet{DashtiLawStuartVoss2013} and \citet{AgapiouBurgerDashtiHelin2018}) rather than closed balls $\ccBall{x}{\varepsilon}$ (following e.g.\ \citet{Bogachev1998Gaussian}).
	However, \Cref{prop:Open_Closed_Coincide} shows that these two notions yield the same definition of \ac{OM} functionals and global weak modes.
\end{remark}

\begin{remark}[Uniqueness of \ac{OM} functionals]%
	\label{remark:UniquenessOMfunctionals}%
	Note that \ac{OM} functionals are at best unique up to the addition of real constants.
	Whenever we talk about $\Gamma$-convergence and equicoercivity of sequences of \ac{OM} functionals, which are at the core of this work, we always mean the existence of representatives that fulfil these properties.
	Further, whenever we apply results that require both $\Gamma$-convergence and equicoercivity (such as \Cref{thm:fundamental_Gamma}), we need to make sure that the \emph{same} representatives can be chosen for both properties.
\end{remark}

\begin{remark}[\ac{OM} functionals and changes of metric]%
	\label{remark:OM_and_change_of_metric}%
	Unfortunately, the choice of metric on a space $X$ can affect the \ac{OM} functional of a measure $\mu$ on $X$, even beyond the non-uniqueness alluded to in \Cref{remark:UniquenessOMfunctionals}, and even if the two metrics are Lipschitz equivalent.
	An explicit example of this is furnished by the finite measure $\mu$ of \citet[Example~5.6]{LieSullivan2018};
	see \Cref{example:OM_and_change_of_metric} for details.
\end{remark}

\subsection{Modes and MAP estimators}

In finite-dimensional spaces, for probability measures that are either purely discrete or possess a continuous Lebesgue density, modes (as points of maximum probability) are easily defined as being global maximisers of the probability mass function or probability density function as appropriate.
For probability measures on infinite-dimensional spaces, however, the situation is more delicate as there is no infinite-dimensional analogue of Lebesgue measure to serve as a uniform reference.
Therefore, it has become common to define modes by examining the masses of norm balls in the small-radius limit.
The following definition of a \emph{strong mode} is a slight generalisation of the definition of a \emph{\ac{MAP} estimator} for a Bayesian posterior measure on a normed space as given by \citet{DashtiLawStuartVoss2013}.

\begin{definition}
	\label{defn:strong_mode}
	Let $X$ be a metric space.
	A \defterm{strong mode} of $\mu \in \prob{X}$ is any $u \in X$ satisfying
	\begin{equation}
		\label{eq:strong_mode}
		\lim_{r \searrow 0} \frac{\mu(\cBall{u}{r})}{M_{r}}=1 ,
	\end{equation}
	where $M_{r} \defeq \sup_{w\in X} \mu(\cBall{w}{r}) \in (0, 1]$.
	Since $\mu(\cBall{u}{r}) \leq M_{r}$, the ratio inside the limit in \eqref{eq:strong_mode} is at most one, and so it is equivalent to define a strong mode as being any $u \in X$ for which
	\[
		\lim_{r \searrow 0} \frac{\mu(\cBall{u}{r})}{M_{r}} \geq 1
		\text{ or }
		\liminf_{r \searrow 0} \frac{\mu(\cBall{u}{r})}{M_{r}} \geq 1
		\text{ or }
		\limsup_{r \searrow 0} \frac{M_{r}}{\mu(\cBall{u}{r})} \leq 1
		.
	\]
\end{definition}

A related notion of mode for a measure is the \emph{weak mode} or \emph{weak \ac{MAP} estimator} \citep[Definition 4]{HelinBurger2015}, which are points that dominate all other points within an affine subspace, in terms of small ball probabilities.
Since we are only interested in \emph{global} weak modes, we simplify the definition slightly and at the same time generalise this concept to metric spaces. The original definition relies on subtraction and thus only applies in the case of linear spaces.

\begin{definition}
	\label{defn:global_weak_mode}
	For a metric space $X$, a \defterm{global weak mode} of $\mu \in \prob{X}$ is any $u \in \supp (\mu)$ satisfying, for any point $u' \in X$,
	\begin{equation}
		\label{eq:global_weak_mode}
		\limsup_{r \searrow 0} \frac{\mu (\cBall{u'}{r})}{\mu (\cBall{u}{r})} \leq 1 .
	\end{equation}
\end{definition}

\begin{remark}
	\label{rem:limsup_vs_lim}
\change{The definition \eqref{eq:global_weak_mode} differs further from that of \citet[Definition~4]{HelinBurger2015} in that we use a limit superior instead of a limit.
We suspect that this is the way it was intended to be defined --- if a point dominates every other point in terms of small ball probabilities (in the sense that the ratio in \eqref{eq:global_weak_mode} becomes $\leq 1$ for sufficiently small $r > 0$), it should be called a weak mode.
This suspicion is based on \citet[Lemma 3]{HelinBurger2015}, where the authors prove that every strong mode is a weak mode, which is clearly a desirable property given the terminology ``strong'' and ``weak'', but their proof is incorrect given their original definition.
The reason is that the ratio in \eqref{eq:global_weak_mode} can drop below $\leq 1$ as $r\searrow 0$ for certain points $u,u'\in X$ without the limit existing (e.g.\ it might oscillate between $0$ and $\tfrac{1}{2}$), such that $u$ cannot be a weak mode in the original definition, while it can still be a strong mode.
\cref{example:Minimiser_of_OM_not_weak_mode} shows that such oscillations can occur and a slight modification provides a concrete counterexample to \citet[Lemma 3]{HelinBurger2015}.
}
\end{remark}

\begin{lemma}
	\label{lem:strong_mode_implies_global_weak_mode}
	Let $\mu\in \prob{X}$ and let $u\in\supp(\mu)$ be a strong mode of $\mu$. Then $u$ is a global weak mode of $\mu$.
\end{lemma}

\begin{proof}
	Since $u\in \supp(\mu)$ is a strong mode of $\mu$, we obtain, for any point $u' \in X$,
	\[
		\limsup_{r \searrow 0} \frac{\mu (\cBall{u'}{r})}{\mu (\cBall{u}{r})}
		\leq
		\limsup_{r \searrow 0} \frac{M_{r}}{\mu (\cBall{u}{r})}
		=
		\lim_{r \searrow 0} \frac{M_{r}}{\mu (\cBall{u}{r})}
		=
		1.
	\]
\end{proof}

Sufficient conditions for the converse implication are given by \citet{LieSullivan2018}.
The relationship between modes and \ac{OM} functionals will be examined in \Cref{sec:modes_and_OM_functionals}.

\subsection{Generalised inverses}

We adopt the following definition of the Moore--Penrose pseudoinverse of an operator \citep[Definition~2.2]{EnglHankeNeubauer1996}:

\begin{definition}
	\label{def:pseudo_inverse_operator}
	For a bounded linear operator $A \colon X \to Y$ between Hilbert spaces $X$ and $Y$, the \defterm{Moore--Penrose pseudoinverse} $A^{\dagger}$ of $A$ is the unique extension of $(A|_{(\ker A)^{\perp}})^{-1}$ to a (generally unbounded) linear operator $A^{\dagger} \colon \range A \oplus (\range A)^{\perp} \to X$ subject to the restriction that $\ker A^{\dagger} = (\range A)^{\perp}$.
\end{definition}

\begin{remark}
    \label{rem:minimum_norm_properties_pseudoinverse}
	For $y \in \range A \oplus (\range A)^{\perp}$,
	\[
		A^{\dagger} y
		=
		\argmin \Set{ \norm{ x }_{X} }{ \vphantom{\big|}x \text{ minimises } \norm{ A x - y } } .
	\]
	In particular, for $y \in \range A$, $A^{\dagger} y$ is the minimum-norm solution of $A x = y$ \citep[Theorem~2.5]{EnglHankeNeubauer1996}.
\end{remark}

\begin{remark}
	\label{remark:SquareRootOfSPSDOperatorAndItsPseudoInverse}
	For a \ac{SPSD} and compact operator $C = \sum_{n \in \Naturals} \sigma_{n}^{2}\, e_{n} \otimes e_{n}\colon X \to X$ on a Hilbert space $X$, $(e_{n})_{n\in\Naturals}$ being an orthonormal system in $X$ and $\sigma_{n} \geq 0$ for each $n\in\Naturals$, we denote the \ac{SPSD} operator square root of $C$ by $C^{1/2}$ and furthermore set
	\[
		C^{\dagger/2}
		\defeq
		(C^{1/2})^{\dagger}
		=
		\sum_{n \in \Naturals\, : \, \sigma_{n} \neq 0} \sigma_{n}^{-1} \, e_{n} \otimes e_{n} .
	\]
	Note that $(C^{\dagger})^{1/2}$ can differ from $(C^{1/2})^{\dagger}$ since it may have a smaller domain.
\end{remark}

\section{Modes, Onsager--Machlup functionals, and their convergence}
\label{sec:modes_and_OM_functionals}

The purpose of this section is to firmly establish the intuitively plausible relationship between the modes of a probability measure $\mu$ and its \ac{OM} functional $I$, namely that the global weak modes of $\mu$ are exactly the global minimisers of $I$.
Once this is done, it is a relatively simple matter to give sufficient conditions for the global weak modes of a sequence of measures to converge to the global weak modes of a limiting measure:
$\Gamma$-convergence and equicoercivity of the associated \ac{OM} functionals.

\begin{proposition}[Global weak modes and \ac{OM} functionals]
	\label{prop:weak_mode_OM_minimiser}
	Let $X$ be a metric space and let $I \colon E \to \Reals$ be an \ac{OM} functional for $\mu \in \prob{X}$, defined on a nonempty subset $E \subseteq X$ with property $M(\mu,E)$.
    Then $u \in E$ is a global weak mode of $\mu$ if and only if $u$ is a minimiser of the extended \ac{OM} functional $I\colon X \to \eReals$.
\end{proposition}

\begin{proof}
	By property $M(\mu,E)$ and \Cref{lem:vanishing_small_ball_prob}\ref{item:vanishing_small_ball_prob_not_weak_mode}, any global weak mode of $\mu$ must lie in $E$, and in addition any minimiser of $I\colon X \to \eReals$ must also lie in $E$, where $E$ is the set on which $I$ takes real values.
	Let $u\in E$ be arbitrary.
	Then $\lim_{r \searrow 0} \frac{\mu(\cBall{u'}{r})}{\mu(\cBall{u}{r})}$ exists for any $u' \in E$ by definition of $I$ being the \ac{OM} functional, and the same limit exists and equals 0 for $u' \in X\setminus E$, because property $M(\mu,E)$ holds.
	Thus,
	\begin{align*}
		u \in E \text{ is a global weak mode}
		& \iff
		\text{for all $u' \in X$, } \limsup_{r \searrow 0} \frac{\mu(\cBall{u'}{r})}{\mu(\cBall{u}{r})} \leq 1 \\
		& \iff
		\text{for all $u' \in X$, } \lim_{r \searrow 0} \frac{\mu(\cBall{u'}{r})}{\mu(\cBall{u}{r})} \leq 1 \\
		& \iff
		\text{for all $u' \in X$, } \exp (I(u)-I(u')) \leq 1 \\
		& \iff
		\text{for all $u' \in X$, } I(u) \leq I(u') ,
	\end{align*}
	as claimed.
\end{proof}

Property $M(\mu,E)$ was essential in the above argument in order for points outside $E$ to be treated in a consistent way.
Recall that, in \cref{defn:Onsager--Machlup}, for a given $\mu \in \prob{X}$, we initially defined the \ac{OM} functional of $\mu$ to be a function $I \colon E \to \Reals$.
Only under property $M(\mu,E)$ can we sensibly extend $I$ to a $\eReals$-valued function on $X$ by setting $I(x) \defeq +\infty$ for $x\in X\setminus E$.
The motivation for this extension is that, by \Cref{lem:vanishing_small_ball_prob}\ref{item:vanishing_small_ball_prob_not_weak_mode}, no point of $X \setminus E$ can be a global weak mode for $\mu$, and hence cannot be a strong mode for $\mu$.

Unfortunately, without additional assumptions, an analogous result to	\Cref{prop:weak_mode_OM_minimiser} cannot hold for strong modes, as demonstrated in \Cref{example:Minimiser_of_OM_not_strong_mode}.
\change{The main idea behind the measure $\mu$ constructed therein is that $u=1$ ``dominates'' any other (fixed) point $u' \in X = \Reals$ in the limit $r \searrow 0$, i.e.\ $\lim_{r \searrow 0} \frac{\mu(\cBall{u'}{r})}{\mu(\cBall{u}{r})} \leq 1$, hence $u=1$ is a global weak mode; but for certain arbitrarily small radii $r_{n}$, $n\in\Naturals$, there exist points $u_{n} \in X$ that ``dominate'' $u$ by a margin, in fact $\liminf_{r \searrow 0} \frac{\mu(\cBall{u}{r})}{M_{r}} \leq \tfrac{1}{\sqrt{2}}$, hence $u=1$ cannot be a strong mode.
Moreover, for $E = \Naturals \subseteq X$, property $M(\mu,E)$ holds and an \ac{OM} functional $I_{\mu,E}\colon E \to \Reals$ exists and has $u=1$ as its minimizer.
The construction is based on suitably chosen singularities of the Lebesgue density $\rho$ of $\mu$.
}

The following result, which is an almost immediate consequence of the preceding discussion, provides clear criteria for the convergence of global weak modes along sequences of probability measures.
(Definitions and basic properties of $\Gamma$-convergence, equicoercivity, etc.\ are collected in \Cref{sec:Gamma}.)

\begin{theorem}[$\mathsf{\Gamma}$-convergence and equicoercivity imply convergence of modes]
	\label{thm:convergence_of_modes}
	Let $X$ be a metric space and let, for $n \in \Naturals \cup \{ \infty \}$, $\mu^{(n)} \in \prob{X}$ have \ac{OM} functionals $I^{(n)} \colon E^{(n)} \to \Reals$, such that property $M(\mu^{(n)},E^{(n)})$ is satisfied.
	Extend each $I^{(n)}$ to take the value $+ \infty$ on $X \setminus E^{(n)}$.
	Suppose that the sequence $(I^{(n)})_{n \in \Naturals}$ is equicoercive and $\Gamma$-converges to $I^{(\infty)}$.
	Then, if $u^{(n)}$ is a global weak mode of $\mu^{(n)}$, $n\in\Naturals$, every convergent subsequence of $(u^{(n)})_{n \in \Naturals}$ has as its limit a global weak mode of $\mu^{(\infty)}$.
\end{theorem}

\begin{proof}
	By \Cref{prop:weak_mode_OM_minimiser}, the global weak modes of $\mu^{(n)}$ are precisely the minimisers of the extended version of $I^{(n)}$, $n \in \Naturals \cup \{ \infty \}$.
	The rest follows immediately from the fundamental theorem of $\Gamma$-convergence (\Cref{thm:fundamental_Gamma}).
\end{proof}

\change{It is instructive to reconsider the earlier \Cref{example:MAP_counterexamples}\ref{item:pointwise_convergent_density} in light of \Cref{thm:convergence_of_modes}.
The problem in \Cref{example:MAP_counterexamples}\ref{item:pointwise_convergent_density} --- in which the unique modes of the measures $\mu^{(n)}$ fail to cluster at the unique mode of the limiting measure $\mu^{(\infty)}$ --- can now be recognised as being due to the fact that although pointwise convergence of Lebesgue densities and \ac{OM} functionals holds, $\Gamma$-convergence does not.
Therefore, \Cref{thm:convergence_of_modes} does not apply to that example and there is no reason for modes to converge in this case.}

\Cref{thm:convergence_of_modes} is, of course, a highly general result.
For it to be useful in specific situations, one must prove property $M (\mu^{(n)}, E^{(n)})$ and identify the form of the \ac{OM} functional $I^{(n)}$ for every $n \in \Naturals$.
In addition, one must verify both the $\Gamma$-convergence and equicoercivity properties of the sequence $(I^{(n)})_{n \in \Naturals}$.
In the next section, we do this for Gaussian measures and Besov-1 probability measures, which are commonly used as priors in the context of \ac{BIP}s.

\section{\texorpdfstring{$\boldsymbol{\mathsf{\Gamma}}$}{Gamma}-convergence of Onsager--Machlup functionals for Gaussian and Besov-1 priors}
\label{sec:convergence_of_OM_functionals}

This section illustrates the preceding general theory of convergence of modes via $\Gamma$-convergence of \ac{OM} functionals by means of two key examples, namely Gaussian and Besov $B_{1}^{s}$ measures, both of which commonly arise as prior distributions in \ac{BIP}s.
Besov $B_{p}^{s}$-priors with $1 \leq p \leq 2$, Cauchy priors, and more general product measures are treated in a unified way in Part~II of this paper \citep{AyanbayevKlebanovLieSullivan2021_II}.
The convergence of modes (\ac{MAP} estimators) for \emph{posterior} distributions will be discussed in \Cref{sec:MAP_in_BIP}.

\subsection{Gaussian measures}
\label{sec:Gamma_Gaussian}

As a natural first case, we consider the $\Gamma$-convergence of the \ac{OM} functionals of Gaussian measures --- and we call attention to the fact that we consider Gaussian measures with possibly indefinite covariance operators.
It is almost folklore that the \ac{OM} functional of a Gaussian measures is half the square of the associated Cameron--Martin norm;
a precise formulation of this result is the following.

\begin{theorem}[\ac{OM} functional of a Gaussian on a separable Banach space]%
	\label{thm:OM_Gaussian_sep_Banach}%
	Let $\mu$ be a centred Gaussian measure on a separable Banach space $X$.
	Let $H(\mu)$ be the Cameron--Martin space of $\mu$, with Cameron--Martin norm $\norm{ \quark }_{H(\mu)}$.
	Then, for all $h, k \in H(\mu)$,
	\begin{equation}
		\lim_{r \searrow 0} \frac{ \mu ( \cBall{h}{r} ) }{ \mu ( \cBall{k}{r} ) } = \exp \left( \tfrac{1}{2} \norm{ k }_{H(\mu)}^{2} - \tfrac{1}{2} \norm{ h }_{H(\mu)}^{2} \right) .
	\end{equation}
	In particular, the \ac{OM} functional for $\mu$ on the Cameron--Martin space $H(\mu)$ is half the square of the Cameron--Martin norm.
\end{theorem}

\begin{proof}
	This is a special case of \citep[Corollary~4.7.8]{Bogachev1998Gaussian}, in which the cylindrical $\sigma$-algebra $\cylindrical{X}$ and the Borel $\sigma$-algebra $\Borel{X}$ coincide by the separability of $X$, the measurable seminorm $q$ is the ambient norm $\norm{ \quark }_{X}$, the $q$-ball $V_{r}$ is the ball $\cBall{0}{r} \in \Borel{X}$, and the projection $\pi_{q}$ is the identity due to the definiteness of $q (\quark)= \norm{ \quark }_{X}$.
	Note that \citet{Bogachev1998Gaussian} works with closed balls, this difference being inconsequential in view of \Cref{prop:Open_Closed_Coincide}.
\end{proof}

\begin{corollary}
	\label{cor:OMfunctionalForCenteredGaussians}
	Let $\mu = \Normal (0,C)$ be a centered Gaussian measure on a separable Hilbert space $X$, where the covariance $C$ is interpreted as an \ac{SPSD} operator on $X$.
	Then the (extended) \ac{OM} functional $I_{\mu} \colon X \to \eReals$ of $\mu$ is given by
	\[
	I_{\mu}(u)
	=
	\begin{cases}
	\tfrac{1}{2} \norm{C^{\dagger/2} u}_{X}^{2} & \text{for } u\in H(\mu) = \range C^{1/2},
	\\
	+\infty & \text{otherwise.}
	\end{cases}
	\]
\end{corollary}

\begin{proof}
	By \citep[Section~2.3, p.~49]{Bogachev1998Gaussian}, the reproducing kernel Hilbert space $X_{\mu}^{\ast} \defeq \overline{ X^{\ast} }^{L^{2}(\mu)}$ of $\mu$ can be identified with the weighted Hilbert space of sequences
	\begin{equation}
		\label{equ:SpaceOfSequencesCorrespondingToRKHS}
		\ell_{C}^{2}
		\defeq
		\Set{ x = (x_{n}) }{ \sum_{n\in\Naturals} \sigma_{n}^{2} x_{n}^{2} < \infty },
		\qquad
		\norm{x}_{\ell_{C}^{2}} \defeq \left( \sum_{n\in\Naturals} \sigma_{n}^{2} x_{n}^{2} \right)^{1/2}.
	\end{equation}
	Further, after extending $C$ naturally to $X_{\mu}^{\ast}$, the Cameron--Martin space coincides with the image of $X_{\mu}^{\ast}$ under $C$, i.e.\ $H(\mu) = C(X_{\mu}^{\ast})$.

	Now let $C = \sum_{n \in \Naturals} \sigma_{n}^{2} e_{n} \otimes e_{n}$, $\sigma_{n} \geq 0$, be the eigenvalue decomposition of its covariance operator $C$ with complete orthonormal system $(e_{n})_{n\in\Naturals}$ and let $u = \sum_{n \in \Naturals} u_{n} e_{n} \in H(\mu)$.
	Since $H(\mu) = C(X_{\mu}^{\ast})$, there exists $x = (x_{n})_{n \in \Naturals} \in \ell_{C}^{2}$, such that $u_{n} = \sigma_{n}^{2} x_{n}$ for all $n\in\Naturals$, and, by \citep[Lemma~2.4.1]{Bogachev1998Gaussian} and \Cref{remark:SquareRootOfSPSDOperatorAndItsPseudoInverse},
	\[
		\norm{u}_{H(\mu)}^{2}
		=
		\norm{x}_{\ell_{C}^{2}}^{2}
		=
		\sum_{n \in \Naturals} \sigma_{n}^{2} x_{n}^{2}
		=
		\sum_{n \in \Naturals\, : \, \sigma_{n} \neq 0} \frac{u_{n}^{2}}{\sigma_{n}^{2}}
		=
		\norm{C^{\dagger/2} u}_{X}^{2}.
	\]
	The claim now follows from \Cref{thm:OM_Gaussian_sep_Banach} and from \citep[Lemma~3.7]{DashtiLawStuartVoss2013}, where property $M(\mu,H(\mu))$ is established.\footnote{The statement of \citep[Lemma~3.7]{DashtiLawStuartVoss2013} may be understood as a weaker statement than needed for the $M(\mu,H(\mu))$ ($\liminf$ in place of $\lim$ in \eqref{eq:liminf_small_ball_prob}).
		However, their proof clearly shows the stronger statement as we use it here.
		Further, as pointed out by \citet{Wacker2020}, their proof is not valid in the Banach space setting, but works in the Hilbert space setting, which is sufficient for our purposes.}
\end{proof}

\begin{remark}
	Note that the notation in \citep[Section~2.3, p.~49]{Bogachev1998Gaussian} is slightly imprecise, since the space $\ell_{C}^{2}$ in \eqref{equ:SpaceOfSequencesCorrespondingToRKHS} is, in general, only a pre-Hilbert space (and $\norm{\quark}_{\ell_{C}^{2}} $ is just a seminorm).
	To be rigorous, one would need to consider the quotient space of $\ell_{C}^{2}$ after factoring out the subspace $\{ x \mid \norm{x}_{\ell_{C}^{2}} = 0 \}$.
	This detail has no influence on the proof of \Cref{cor:OMfunctionalForCenteredGaussians}.
\end{remark}

\begin{corollary}
	\label{cor:OMfunctionalForUncenteredGaussians}
	Let $\mu_{0} = \Normal (0,C)$ be a centered Gaussian measure on a separable Hilbert space $X$, where the covariance $C$ is interpreted as an \ac{SPSD} operator on $X$, and $\mu = \Normal (m,C)$.
	Then the \ac{OM} functional $I_{\mu} \colon X \to \eReals$ of $\mu$ is given by
	\[
		I_{\mu}(u)
		=
		\begin{cases}
			\tfrac{1}{2} \norm{ (u-m) }_{H(\mu_{0})}^{2}
			=
			\tfrac{1}{2} \norm{C^{\dagger/2} (u-m)}_{X}^{2} & \text{for } (u-m)\in H(\mu_{0}) = \range C^{1/2},
			\\
			+\infty & \text{otherwise.}
		\end{cases}
	\]
\end{corollary}

\begin{proof}
	This follows directly from \Cref{thm:OM_Gaussian_sep_Banach} and \Cref{cor:OMfunctionalForCenteredGaussians}.
\end{proof}

We now give the main result of this section, that the strong (norm) convergence of means and covariance operators of Gaussian measures is sufficient to ensure that their associated \ac{OM} functionals are $\Gamma$-convergent and equicoercive.

\begin{theorem}[$\mathsf{\Gamma}$-convergence and equicoercivity of \ac{OM} functionals for Gaussian measures]
	\label{thm:GammaConvergenceForGaussianOM}
	Let $X$ be a separable Hilbert space and $\mu^{(n)} = \Normal (m^{(n)}, C^{(n)})$ and $\mu = \Normal (m, C)$ be Gaussian measures on $X$ such that $m^{(n)} \to m$ in $X$ and $C^{(n)} \to C$ with respect to the operator norm.
	Then $I_{\mu} = \Gammalim_{n \to \infty} I_{\mu^{(n)}}$.
	Furthermore, the sequence $(I_{\mu^{(n)}})_{n\in\Naturals}$ is equicoercive.
\end{theorem}

\change{%
\begin{remark}
	Since all the Gaussian \ac{OM} functionals $I_{\mu^{(n)}}$ are quadratic forms, and homogeneity is preserved by $\Gamma$-limits \citep[Proposition~2.13]{Braides2006}, it is not surprising that $\Gammalim_{n \to \infty} I_{\mu^{(n)}}$ is quadratic --- the point here is to check that the quadratic forms $\Gammalim_{n \to \infty} I_{\mu^{(n)}}$ and $I_{\mu}$ agree, and moreover with careful attention to the possibility of indefinite covariances.
\end{remark}
}

\begin{proof}[Proof of \Cref{thm:GammaConvergenceForGaussianOM}]
	Let $A \defeq C^{1/2}$ and $A_{n} \defeq (C^{(n)})^{1/2}$.
	Further, let $(e_{k})_{k\in\N}$ be an orthonormal eigenbasis of $A$, $A = \sum_{k\in\N} \sigma_{k}\, e_{k}\otimes e_{k}$ with $\sigma_{k} \geq 0$, and, for any vector $w\in X$, let $w_{k} \defeq \innerprod{w}{e_{k}}_{X}$ denote its $k$\textsuperscript{th} component in that basis.

	Let $(u^{(n)})_{n\in\Naturals}$ be a sequence in $X$ that converges to $u\in X$. If $\liminf_{n\to\infty} I_{\mu^{(n)}}(u^{(n)}) = \infty$, then there is nothing to prove.
	Therefore, define $\mathcal{I} \defeq \liminf_{n\to\infty} I_{\mu^{(n)}}(u^{(n)}) \in \Reals$.
	There exists a subsequence of $(u^{(n)})_{n \in \Naturals}$, which for simplicity we also denote by $(u^{(n)})_{n \in \Naturals}$, such that $u^{(n)} - m^{(n)} \in \range A_{n}$ for each $n\in\N$ and $I_{\mu^{(n)}}(u^{(n)}) \xrightarrow[n\to\infty]{} \mathcal{I}$ \change{(note that $I_{\mu^{(n)}}(u^{(n)}) = \infty$ unless $u^{(n)} - m^{(n)} \in \range A_{n}$)}.

	Now let $\varepsilon > 0$ and $v^{(n)}\defeq A_{n}^{\dagger} (u^{(n)} - m^{(n)})$, $n \in \Naturals$.
	Without loss of generality (possibly, by a further thinning of the subsequence) and using \Cref{cor:OMfunctionalForUncenteredGaussians}, we may assume
	$\tfrac{1}{2} \norm{v^{(n)}}_{X}^{2} = I_{\mu^{(n)}}(u^{(n)}) \leq \mathcal{I} + \varepsilon$ for each $n\in\N$.
	Define $\mathcal{K} \defeq \{ k\in\N \mid \sigma_{k} > 0 \}$ and the sequences $\hat{v} = (v_{k})_{k\in\N}$ and $\round{\hat{v}^{(n)}} = (\round{v_{k}^{(n)}})_{k\in\N}$, $n\in\N$, by
	\[
		v_{k}
		\defeq
		\begin{cases}
			\frac{u_{k} - m_{k}}{\sigma_{k}} & \text{if } k\in \mathcal{K},
			\\
			0 & \text{otherwise},
		\end{cases}
		\qquad
		\round{v_{k}^{(n)}}
		\defeq
		\begin{cases}
			v_{k}^{(n)} & \text{if } k\in \mathcal{K},
			\\
			0 & \text{otherwise}.
		\end{cases}
	\]

	To prove the $\Gamma$-$\liminf$ inequality, and with \Cref{cor:OMfunctionalForUncenteredGaussians} in mind, we must show that
	\begin{enumerate}[label=(\roman*)]
		\item
		\label{item:vInX}
		$\hat{v} \in \ell^{2}$ and therefore $v \defeq \sum_{k \in \Naturals} v_{k} e_{k} \in X$;
		\item
		\label{item:AvEqualsu}
		$A v = u-m$ and therefore $u-m\in\range{A}$;
		\item
		\label{item:vHasCorrectUpperBound}
		$\tfrac{1}{2}\norm{v}_{X}^{2} \leq \mathcal{I} + \varepsilon$ and therefore
		\[
			I_{\mu}(u)
			=
			\tfrac{1}{2}\norm{C^{\dagger/2} (u-m)}_{X}^{2}
			\leq
			\tfrac{1}{2}\norm{v}_{X}^{2}
			\leq
			\mathcal{I}
			=
			\liminf_{n\to\infty} I_{\mu^{(n)}}(u^{(n)}),
		\]
		by using the fact that $\varepsilon > 0$ is arbitrary and by using \Cref{rem:minimum_norm_properties_pseudoinverse}.
	\end{enumerate}

	Since $\norm{A_{n} - A} \to 0$, $\norm{u^{(n)} - u}_{X} \to 0$ and $\norm{v^{(n)}}_{X} \leq M_{\varepsilon} \defeq \sqrt{2\mathcal{I} + 2\varepsilon}$ for all $n\in\N$, we obtain
	\begin{align*}
		\norm{u - m - A v^{(n)}}_{X}
		&\leq
		\norm{(u-m) - (u^{(n)} - m^{(n)})}_{X} + \norm{u^{(n)} - m^{(n)} - A v^{(n)}}_{X}
		\\
		&=
		\norm{u - u^{(n)}}_{X} + \norm{m - m^{(n)}}_{X} + \norm{(A_{n} - A) v^{(n)}}_{X}
		\\
		&\leq
		\norm{u - u^{(n)}}_{X} + \norm{m - m^{(n)}}_{X} + \norm{ A_{n} - A } \norm{v^{(n)}}_{X}
		\\
		&\xrightarrow[n\to\infty]{}
		0.
	\end{align*}

	The above convergence implies componentwise convergence: $\absval{u_{k} - m_{k} - \sigma_{k} v_{k}^{(n)}} \xrightarrow[n\to\infty]{} 0$ for each $k\in\N$ or, equivalently, $\round{v_{k}^{(n)}} \xrightarrow[n\to\infty]{} v_{k}$ for each $k\in\N$ and $u_{k} - m_{k} = 0$ for all $k \notin \mathcal{K}$.
	Since, for each $n\in\N$, $\round{\hat{v}^{(n)}} \in \ell^{2}$ with $\norm{\round{\hat{v}^{(n)}}}_{\ell^{2}} \leq \norm{v^{(n)}}_{X} \leq M_{\varepsilon}$, \Cref{lemma:TechnicalLemmaOnConvergenceInl2} implies that $\hat{v} \in \ell^{2}$ and $v \in X$ with $\norm{\hat{v}}_{\ell^{2}} = \norm{v}_{X} \leq M_{\varepsilon}$, proving \ref{item:vInX} and \ref{item:vHasCorrectUpperBound}.
	Since $u_{k} - m_{k} = 0$ for all $k \notin \mathcal{K}$, we obtain $A v = \sum_{k\in \mathcal{K}} (u_{k} - m_{k}) e_{k} = \sum_{k\in \Naturals} (u_{k} - m_{k}) e_{k} = u - m$, proving \ref{item:AvEqualsu} and finalising the proof of the $\Gamma$-$\liminf$ inequality.

	For the $\Gamma$-$\limsup$ inequality, first note that, if $u-m \notin \range A$, then $I_{\mu}(u) = \infty$, and there is nothing to prove since we may choose $u^{(n)} \defeq u$ for all $n \in \Naturals$.
	If $u-m \in \range A$, let $v \defeq A^{\dagger} (u-m)$ and $u^{(n)} \defeq m^{(n)} + A_{n} v$.
	Then $\norm{u^{(n)} - u}_{X} \leq \norm{m^{(n)} - m}_{X} + \norm{A_{n} - A} \norm{v}_{X} \xrightarrow[n\to\infty]{} 0$.
	Since $v$ is some solution of $A_{n} x = u^{(n)} - m^{(n)}$ and $A_{n}^{\dagger} (u^{(n)} - m^{(n)})$ is its minimum norm solution (cf.\ \Cref{rem:minimum_norm_properties_pseudoinverse}), \Cref{cor:OMfunctionalForUncenteredGaussians} implies
	\[
		I_{\mu}(u)
		=
		\tfrac{1}{2}\norm{v}_{X}^{2}
		\geq
		\tfrac{1}{2}\norm{A_{n}^{\dagger} (u^{(n)} - m^{(n)})}_{X}^{2}
		=
		I_{\mu^{(n)}}(u^{(n)})
	\]
	for each $n\in\Naturals$, finalising the proof of the $\Gamma$-$\limsup$ inequality.

	In order to prove equicoercivity of the sequence $(I_{\mu^{(n)}})_{n\in\Naturals}$, let $t\in \Reals$ and
	\begin{align*}
		K_{t}
		& \defeq
		\bigcup_{n\in \Naturals} K_{t}^{(n)},
		\\
		K_{t}^{(n)}
		& \defeq
		I_{\mu^{(n)}}^{-1} ([-\infty , t])
		=
		\Set{ u \in m^{(n)} + \range A_{n} }{ \tfrac{1}{2} \norm{A_{n}^{\dagger} (u-m^{(n)})}_{X}^{2} \leq t }
		=
		m^{(n)} + A_{n} \overline{B}_{\sqrt{2t}}(0),
	\end{align*}
	where we used \Cref{cor:OMfunctionalForUncenteredGaussians}.
	We will now show that $K_{t}$ is (sequentially) precompact.
	To this end, let $(u^{(\nu)})_{\nu \in \Naturals}$ be a sequence in $K_{t}$.
	If $u^{(\nu)} \in K_{t}^{(n)}$ infinitely often for some $n \in \Naturals$, there is nothing to prove, since $A_{n}$ is a compact operator for each $n \in \Naturals$.
	Otherwise, there exist subsequences $(u^{(\nu_{j})})_{j \in \Naturals}$ and $(\mu^{(n_{j})})_{j \in \Naturals}$ such that $u^{(\nu_{j})} \in K_{t}^{(n_{j})}$ for each $j \in \Naturals$.
	Hence, $u^{(\nu_{j})} - m^{(n_{j})} \in \range A_{n_{j}}$ for each $j\in \Naturals$ and the points $v^{(j)} \defeq A_{n_{j}}^{\dagger} (u^{(\nu_{j})} - m^{(n_{j})})$ are uniformly bounded, $\norm{v^{(j)}}_{X} \leq \sqrt{2t}$ for $j \in \Naturals$.
	Since $A$ is a compact operator, the sequence $(w^{(j)})_{j \in \Naturals}$ given by $w^{(j)} \defeq A v^{(j)}$ has a subsequence that converges to some element $w \in X$. For simplicity, we denote this subsequence by $(w^{(j)})_{j \in \Naturals}$.
	It follows that
	\begin{align*}
		\norm{u^{(\nu_{j})} - m - w}_{X}
		& \leq
		\norm{u^{(\nu_{j})} - m^{(n_{j})} - w^{(j)}}_{X}
		+ \norm{m^{(n_{j})} - m}_{X}
		+ \norm{w^{(j)} - w}_{X}
		\\
		&\leq
		\underbrace{\norm{A_{n_{j}} - A}}_{\to 0}
		\underbrace{\norm{v^{(j)}}_{X}}_{\leq \sqrt{2t}}
		+ \underbrace{\norm{m^{(n_{j})} - m}_{X}}_{\to 0}
		+ \underbrace{\norm{w^{(j)} - w}_{X}}_{\to 0}
		\\
		&\xrightarrow[j\to\infty]{} 0,
	\end{align*}
	and so $(u^{(\nu)})_{\nu \in \Naturals}$ has a convergent subsequence.
	Hence, $\overline{K}_{t}$ is compact with $I_{\mu^{(n)}}^{-1} ([-\infty , t]) \subseteq \overline{K}_{t}$ for each $n\in\Naturals$, finalising the proof of equicoercivity.
\end{proof}

The following corollary is a direct consequence of \Cref{thm:convergence_of_modes,thm:GammaConvergenceForGaussianOM}:

\begin{corollary}
	Let $X$, $\mu$, $(\mu^{(n)})_{n\in\Naturals}$ be as in \Cref{thm:GammaConvergenceForGaussianOM}.
	If $u^{(n)}$ is a global weak mode of $\mu^{(n)}$, $n\in\Naturals$, then every convergent subsequence of $(u^{(n)})_{n \in \Naturals}$ has as its limit a global weak mode of $\mu$.
\end{corollary}

\subsection{$B_{1}^{s}$-Besov measures}
\label{sec:Gamma_Besov1}

We now establish analogous results to those of the previous section for the class of Besov-1 measures. Besov-1 measures and Gaussian measures on infinite-dimensional spaces are analogous to Laplace distributions and normal distributions on $\R$.
Besov-1 measures have been used as sparsity-promoting or edge-preserving priors\footnote{\change{Strictly speaking, regularisation using the Besov-1 norm promotes edge-preservation for the MAP estimator but not for samples from the full posterior distribution.}} in inverse problems \citep{AgapiouBurgerDashtiHelin2018,DashtiHarrisStuart2012,LassasSaksmanSiltanen2009}.

Throughout this subsection, we use the following notation:\footnote{Typically, Besov measures are introduced on the space $L^2(\Torus^{d})$;
the same construction that we use for the components of a random sequence in $\Reals^{\Naturals}$ is used for the components of a random Fourier or wavelet expansion in $L^2(\Torus^{d})$.
In our definition, the dimension $d$ becomes superfluous and one could work with $\tilde{s} \defeq s / d$, but we continue to use the classical notation in order to reduce confusion.}

\begin{assumption}
	\label{notation:General_assumption_Besov}
	Let $s \in \Reals$, $d\in\Naturals$, \change{$\eta > 0$, $t \defeq s - d (1+\eta)$}
	and assume that $\tau \defeq (s/d + 1/2)^{-1} > 0$.
	The parameter $s$ is thought of as a ``smoothness parameter'' and $d$ as a ``spatial dimension''.
	Define $\gamma_{0} \defeq 1$ and $\gamma,\delta \in \Reals^{\Naturals}$ by
	\change{
	\[
		\gamma_{k}
		\defeq
		k^{1 - \tfrac{1}{\tau}},
		\qquad
		\delta_{k}
		\defeq
		k^{2 + \eta - \tfrac{1}{\tau}},
		\qquad
		k\in\Naturals,
	\]
	}
	and let $\mu_{k} \in \prob{\R}$ for $k \in \Naturals \cup \{ 0 \}$ have the Lebesgue density
	\change{
	\[
		\frac{\rd \mu_{k}}{\rd u} (u) = Z_{1}\, \gamma_{k}^{-1} \exp(-\absval{ u / \gamma_{k}}),
		\qquad
		Z_1\defeq \left(\int_{\Reals}\exp(-\absval{x}) \, \rd x\right)^{-1} = \frac{1}{2 \Gamma (2)}.
	\]
	}
\end{assumption}

We define the Besov measure $\change{B_{1}^{s}}$ as follows, using notation that is an adaptation of that of \citet{DashtiHarrisStuart2012} and \citet{AgapiouBurgerDashtiHelin2018}.

\begin{definition}[Sequence space Besov measures and Besov spaces]
	\label{def:Besov_space_and_measure_sequence}
	Using \Cref{notation:General_assumption_Besov}, we call $\mu \defeq \bigotimes_{k \in \Naturals} \mu_{k}$ a (\defterm{sequence space}) \defterm{Besov measure} on $\Reals^{\Naturals}$ and write $\change{ B^{s}_{1} \defeq \mu }$.
	The corresponding \defterm{Besov space} is the weighted sequence space $\change{ (X^{s}_{1},\norm{\quark}_{X^{s}_{1}}) \defeq (\ell^{1}_{\gamma},\norm{\quark}_{\ell^{1}_{\gamma}}) }$.
\end{definition}

Since it is the parameter \change{``$p=1$''} that most strongly affects the qualitative properties of the measure, we often refer simply to a \change{``Besov-$1$ measure''} for any measure in the above class, regardless of the values of $s$, $d$, etc.

\begin{lemma}
	\label{lemma:Besov_space_with_full_measure}
	Let \change{$\mu = B^{s}_{1}$} be the Besov measure defined above and \change{$X = X_{1}^{t} = \ell_{\delta}^{1}$}.
	Then $\mu(X) = 1$.
\end{lemma}

\begin{proof}
	This is a restatement of \citet[Lemma~2]{LassasSaksmanSiltanen2009} \change{for particular case $p=1$}.
\end{proof}

From now on we will consider the Besov measure \change{$\mu = B^{s}_{1}$} on the normed spaces \change{$X = X_{1}^{t} = \ell_{\delta}^{1}$}.
This is possible since, by \citep[Lemma~B.1]{AyanbayevKlebanovLieSullivan2021_II}, \change{$\Borel{\ell_{\delta}^{1}} \subseteq \Borel{\Reals^{\Naturals}}$}, where we consider the product topology on $\Reals^{\Naturals}$.

\begin{proposition}
	\label{prop:OM_Besov_1}
	Let $\mu = B^{s}_{1}$ be a $B_{1}^{s}$-Besov measure on the space $X = X_{1}^{t} = \ell_{\delta}^{1}$.
	Then, for $E = X_{1}^{s}=\ell^{1}_{\gamma}$, property $M(\mu, E)$ is satisfied and the \ac{OM} functional $I_{\mu} \colon X \to \eReals$ of $\mu$ is given by
	\begin{equation}
		\label{equ:OM_for_Besov}
		I_{\mu}(u)
		=
		\begin{cases}
			\norm{u}_{X_{1}^{s}} & \text{for } u\in E,
			\\
			\infty & \text{otherwise.}
		\end{cases}
	\end{equation}
\end{proposition}

\begin{proof}
	The \ac{OM} functional formula on $E$ follows from \cite[Theorem~3.9]{AgapiouBurgerDashtiHelin2018}, while property $M(\mu, E)$ follows from \cite[Theorem~4.9]{AyanbayevKlebanovLieSullivan2021_II}. The assumptions of this theorem are fulfilled, given \Cref{def:Besov_space_and_measure_sequence} and \Cref{lemma:Besov_space_with_full_measure}. 
\end{proof}

\begin{remark}
	\Cref{prop:OM_Besov_1} uses and extends \cite[Theorem~3.9]{AgapiouBurgerDashtiHelin2018}.
	The authors write that ``the space $B_{1}^{s}(\mathbb{T}^d)$ here, is the largest space on which the Onsager--Machlup functional is defined''.
	This claim is intuitively true, since $\norm{h}_{X_{1}^{s}}=+\infty$ if $h\notin X_{1}^{s}=E$, and in our notation $X_{1}^{s}$ corresponds to $B_{1}^{s}(\mathbb{T}^{d})$.
	However, one must not a priori exclude the possibility that $I_{\mu}$ can have a different formula outside of $E$.
	Property $M(\mu, E)$ in the above proof is one way to guarantee that the claim is true.
\end{remark}

We now give a $\Gamma$-convergence and equicoercivity result for sequences of Besov-1 measures with converging smoothness parameters.

\begin{theorem}[$\mathsf{\Gamma}$-convergence and equicoercivity of \ac{OM} functionals for Besov-1 measures]
	\label{thm:Gamma_convergence_OM_Besov}
	Let $\mu^{(n)} \defeq B^{s^{(n)}}_{1}$, $n\in\Naturals \cup \{ +\infty \}$, be centered Besov measures such that $s^{(n)}\to s^{(\infty)}$.
	Then there exists $n_{0} \in \Naturals$ such that, for each $n \geq n_{0}$, $\mu^{(n)}(\ell^1_{\delta^{(\infty)}}) = 1$ and we therefore consider these measures on $X = X_{1}^{t^{(\infty)}} = \ell^1_{\delta^{(\infty)}}$ (after dropping the first $n_{0}-1$ measures).
	Then, for the \ac{OM} functionals $I_{\mu^{( n )}} = \norm{\quark}_{X_{1}^{s^{(n)}}} \colon X \to \eReals$, $n\in \{ n' \in \Naturals \mid n'\geq n_{0} \} \cup \{ +\infty \}$, given by \eqref{equ:OM_for_Besov}, the sequence $(I_{\mu^{(n)}})_{n \geq n_{0}}$ is equicoercive and $I_{\mu^{(\infty)}} = \Gammalim_{n \to \infty} I_{\mu^{( n )}}$.
\end{theorem}

\begin{proof}
Since $s^{(n)}\to s^{(\infty)}$ there exists $n_{0} \in \Naturals$ such that, for $n \geq n_{0}$,
$\absval{s^{(n)} - s^{(\infty)}} \leq \tfrac{d \eta^{(\infty)}}{2}$.
Therefore, for $n \geq n_{0}$, we may choose $\eta^{(n)} > 0$ such that
$t^{(n)} = s^{(n)} - d(1+\eta^{(n)}) = s^{(\infty)} - d(1+\eta^{(\infty)}) = t^{(\infty)}$ and consider $\mu^{(n)}$ as a measure on $X = X_{1}^{t^{(\infty)}} = \ell^1_{\delta^{(\infty)}}$ by \Cref{lemma:Besov_space_with_full_measure}.
Without loss of generality, we assume $n_{0} = 1$ from now on in order to simplify notation.
Since $s^{(n)} \geq \overline{s} \defeq s^{(\infty)} - \tfrac{d \eta^{(\infty)}}{2}$,
\[
\gamma_{k}^{(n)}
=
k^{-\tfrac{s^{(n)}}{d} + \tfrac{1}{2}}
\leq
k^{-\tfrac{\overline{s}}{d} + \tfrac{1}{2}}
\eqqcolon
\overline{\gamma}_{k},
\qquad
k,n \in\Naturals,
\]
and, for any $\theta \geq 0$ and $n \in \Naturals$,
\[
I_{\mu^{(n)}}^{-1}([-\infty,\theta])
\defeq
\Set{u\in X}{\sum_{k\in\Naturals} \frac{\absval{u_k}}{\gamma^{(n)}_k}\leq \theta}
\subseteq
\prod_{k \in \Naturals} [ - \overline{\gamma}_{k} \theta , \overline{\gamma}_{k} \theta ]
\eqqcolon
K_{\theta}.
\]
We will now show that $K_{\theta} \subseteq X$ is precompact.
For this purpose, we define the operators
\[
T,T_{m}\colon \ell^{\infty} \to X = \ell_{\delta^{(\infty)}}^{1},
\qquad
T(x) = (\overline{\gamma}_{k}x_{k})_{k\in\Naturals},
\qquad
T_{m}(x) = (\overline{\gamma}_{k}x_{k})_{k=1,\dots,m},
\qquad
m\in\Naturals.
\]
All $T_{m}$ are finite-rank operators that converge to $T$ in the operator norm:
\[
\norm{T_{m} - T}
=
\sup_{\norm{x}_{\ell^{\infty}} \leq 1}
\sum_{k > m} \frac{\overline{\gamma}_{k}}{\delta^{(\infty)}_{k} } \absval{ x_{k}}
\leq
\sum_{k > m} k^{-1-\tfrac{\eta^{(\infty)}}{2}}
\xrightarrow[m\to\infty]{}
0.
\]
Therefore, $T$ is a compact operator and $K_{\theta} = \theta \, T \cBallExtra{0}{1}{\ell^{\infty}}$ is precompact, finalising the proof of equicoercivity.
Note that for $\theta < 0$ there is nothing to prove, since $I_{\mu^{(n)}}^{-1}([-\infty,\theta]) = \emptyset$ for each $n \in \Naturals \cup \{ \infty \}$ in this case.

In order to prove the $\Gamma$-convergence statement, we will first show that $\norm{\gamma^{(n)}-\gamma^{(\infty)}}_{X}\to 0$.
Since the sequences $a^{(n)} \defeq (k^{-1-\eta} \absval{k^{\frac{s^{(\infty)}-s^{(n)}}{d}}-1})_{k\in\Naturals}$ are uniformly bounded by the summable sequence $a = (2k^{-1-\frac{\eta^{(\infty)}}{2}})_{k\in\Naturals}$ (where we used $\absval{s^{(n)} - s^{(\infty)}} \leq \tfrac{d \eta^{(\infty)}}{2}$), the reverse Fatou lemma implies $\norm{\gamma^{(n)}-\gamma^{(\infty)}}_{X}\to 0$ via
\begin{align*}
	\limsup_{n\to\infty}\norm{\gamma^{(n)}-\gamma^{(\infty)}}_{X}
	& =
	\limsup_{n\to\infty} \sum_{k\in\Naturals} k^{-1-\eta} \Absval{k^{\frac{s^{(\infty)}-s^{(n)}}{d}}-1} \\
	& \leq \sum_{k\in\Naturals} \limsup_{n\to\infty} k^{-1-\eta} \Absval{k^{\frac{s^{(\infty)}-s^{(n)}}{d}}-1} \\
	& =	0.
\end{align*}

For the $\Gamma$-$\liminf$ inequality, it follows from $\norm{u^{(n)}-u^{(\infty)}}_{X}\to 0$ and $\norm{\gamma^{(n)}-\gamma^{(\infty)}}_{X}\to 0$ that $\frac{u^{(n)}_k}{\gamma^{(n)}_k}\to\frac{u^{(\infty)}_k}{\gamma^{(\infty)}_k}$ for all $k\in\Naturals$.
Thus, by Fatou's lemma,
\begin{align*}
	I_{\mu^{(\infty)}}(u^{(\infty)})
	& = \|u^{(\infty)}\|_{\ell^{1}_{\gamma^{(\infty)}}} \\
	& = \sum_{k\in\Naturals} \frac{\absval{u^{(\infty)}_k}}{\gamma^{(\infty)}_k} \\
	& = \sum_{k\in\Naturals} \liminf_{n\to\infty}\frac{\absval{u^{(n)}_k}}{\gamma^{(n)}_k} \\
	& \leq \liminf_{n\to\infty} \sum_{k\in\Naturals} \frac{\absval{u^{(n)}_k}}{\gamma^{(n)}_k}  \\
	& =\liminf_{n\to\infty} I_{\mu^{(n)}}(u^{(n)}).
\end{align*}

For the $\Gamma$-$\limsup$ inequality, note that, if $I_{\mu^{(\infty)}}(u^{(\infty)})=\infty$, then there is nothing to prove.
Therefore, let us assume that $I_{\mu^{(\infty)}}(u^{(\infty)})<\infty$, and define $u^{(n)}$ by $u^{(n)}_k = \gamma^{(n)}_k\frac{u^{(\infty)}_k}{\gamma^{(\infty)}_k}$, $k\in\Naturals$.
Then
\[
I_{\mu^{(n)}}(u^{(n)})
=
\sum_{k\in\Naturals} \frac{\absval{u^{(n)}_k}}{\gamma^{(n)}_k}
=
\sum_{k\in\Naturals} \frac{\absval{u^{(\infty)}_k}}{\gamma^{(\infty)}_k}
=
I_{\mu^{(\infty)}}(u^{(\infty)})
<
\infty,
\]
and $\limsup_{n \to \infty} I_{\mu^{(n)}}(u^{(n)}) \leq I_{\mu^{(\infty)}}(u^{(\infty)})$.
Additionally,
\begin{align*}
	\bignorm{ u^{(n)}-u^{(\infty)} }_{X}
	& = \sum_{k\in\Naturals} \frac{\absval{ u^{(n)}_k-u^{(\infty)}_k }}{\delta^{(\infty)}_k}\\
	& =
	\sum_{k\in\Naturals}  \frac{\absval{ u^{(\infty)}_k } }{\gamma^{(\infty)}_k} \frac{\bigabsval{ \gamma^{(n)}_k-\gamma^{(\infty)}_k }}{\delta^{(\infty)}_k}\\
	& \leq
	I_{\mu^{(\infty)}}(u^{(\infty)}) \norm{\gamma^{(n)}-\gamma^{(\infty)}}_{X}\to 0 &  \text{as $n \to \infty$,}
\end{align*}
finalising the proof of the $\Gamma$-$\limsup$ inequality.
\end{proof}

The following corollary is now a direct consequence of \Cref{thm:convergence_of_modes,thm:Gamma_convergence_OM_Besov}:

\begin{corollary}
	Let $X$, $(\mu^{(n)})_{n\in\Naturals\cup \{ \infty \} }$ be as in \Cref{thm:Gamma_convergence_OM_Besov}.
	If, for each $n \in \Naturals$, $u^{(n)}$ is a global weak mode of $\mu^{(n)}$, then every convergent subsequence of $(u^{(n)})_{n \in \Naturals}$ has as its limit a global weak mode of $\mu^{(\infty)}$.
\end{corollary}

\section{Consequences for maximum a posteriori estimation in Bayesian inverse problems}
\label{sec:MAP_in_BIP}

The $\Gamma$-convergence theory of \ac{OM} functionals described in \Cref{sec:convergence_of_OM_functionals} has important consequences for the stability of \ac{MAP} estimators of \acl{BIP}s (\ac{BIP}s), in particular those \ac{BIP}s that use the probability measures considered above as prior distributions.

An inverse problem consists of the recovery of an unknown $u$ from related observational data $y$.
In the Bayesian approach to inverse problems \citep{KaipioSomersalo2005, Stuart2010}, these two objects are treated as coupled random variables $\rv{u}$ and $\rv{y}$ that take values in spaces $X$ and $Y$ respectively.
A priori knowledge about $\rv{u}$ is represented by a prior probability measure $\mu_{0} \in \prob{X}$ and one is given access to a realisation $y$ of $\rv{y}$.
The \emph{solution} of the \ac{BIP} is, by definition, the posterior probability measure $\mu^{y} \in \prob{X}$, i.e.\ the conditional distribution of $\rv{u}$ given that $\rv{y} = y$.
For the sake of space, we omit here all technical discussion of the existence and regularity of this conditional distribution and focus exclusively on the case that $\mu^{y}$ has a Radon--Nikodym derivative with respect to $\mu_{0}$ of the form
\[
	\mu^{y} (\rd u) \propto \exp ( - \Phi (u; y) ) \, \mu_{0} (\rd u)
\]
for some $\Phi \colon X \times Y \to \Reals$.
The function $\Phi$, often called the potential, encodes both the idealised relationship between the unknown and the data and statistical assumptions about any observational noise.
The textbook example is that $X$ is a separable Hilbert or Banach space of functions, $Y = \Reals^{J}$ for some $J \in \Naturals$, and that $\rv{y} = \mathcal{O}(\rv{u}) + \rv{\eta}$ for some deterministic observation map $\mathcal{O} \colon X \to Y$ and additive non-degenerate Gaussian noise $\rv{\eta} \sim \Normal (0, C_{\rv{\eta}})$ that is a priori independent of $\rv{u}$, in which case $\Phi$ is the familiar quadratic misfit
\[
	\Phi (u; y) = \frac{1}{2} \bignorm{ C_{\rv{\eta}}^{-1/2} ( y - \mathcal{O}(u) ) }^{2} .
\]

\change{One} convenient point summary of $\mu^{y}$ is a \ac{MAP} estimator, i.e.\ a point of maximum probability under $\mu^{y}$ in the sense of a maximiser of a small ball probability.
Under the conditions laid out in \Cref{sec:modes_and_OM_functionals}, these points (in the sense of global weak modes) are the minimisers of the \ac{OM} functional of $\mu^{y}$.
However, we note that there are many problems of interest for which a more generalised notion of \ac{MAP} estimator and a correspondingly generalised \ac{OM} functional are needed, particularly problems in which the prior may have bounded support or the potential may take the value $+\infty$ \citep{Clason2019GeneralizedMI}.

Our interest lies in assessing the stability of $\mu^{y}$ (more precisely, the stability of the \ac{MAP} estimators of $\mu^{y}$) in response to the following:
\begin{itemize}
	\item perturbations of the observed data $y$, to be reassured that the posterior is not unduly sensitive to observational errors;
	\item perturbations of the potential\footnote{\change{Of course, a perturbation of the data $y$ induces a perturbation of $\Phi(\quark; y)$.
	Sometimes it is easier to consider data perturbations and potential perturbations separately, and sometimes, as we do in \Cref{thm:Gamma-convergence_of_posteriors}, it is simpler to consider them both as perturbations of the potential.}} $\Phi$, \change{for example} to be reassured that the posterior is not unduly sensitive to numerical approximation of $\mathcal{O}$ by some $\mathcal{O}^{(n)}$ (e.g.\ using a finite element solver to solve a partial differential equation)\change{, or to examine the small-noise limit $C_{\rv{\eta}} \to 0$};
	\item perturbations of the prior $\mu_{0}$, to be reassured that the posterior is not unduly sensitive to prior assumptions, e.g.\ relating to the regularity of $u$.
\end{itemize}
We propose to address this question using the $\Gamma$-convergence results of the previous section.
The classes of measures for whose \ac{OM} functionals explicit $\Gamma$-limits were computed in \Cref{sec:convergence_of_OM_functionals} will serve here as Bayesian prior measures.

Our main result concerns the transfer of convergence properties of sequences of prior \ac{OM} functionals and sequences of potentials to the convergence of posterior \ac{OM} functionals.

\begin{theorem}[Transfer of property $M$, $\mathsf{\Gamma}$-convergence, equicoercivity, and \ac{MAP} estimators]
	\label{thm:Gamma-convergence_of_posteriors}
	Let $X$ be a metric space.
	For each $n \in \Naturals \cup \{ \infty \}$, let $\mu_{0}^{(n)} \in \prob{X}$ and let $\Phi^{(n)} \colon X \to \Reals$ be locally uniformly continuous.
	Suppose that, for each $n \in \Naturals \cup \{ \infty \}$, $Z^{(n)} \defeq \int_{X} e^{- \Phi^{(n)} (x)} \, \mu_{0}^{(n)} (\rd x) \in (0, \infty)$ and set
	\[
		\mu^{(n)} (\rd x) \defeq \frac{1}{Z^{(n)}} e^{- \Phi^{(n)} (x)} \, \mu_{0}^{(n)} (\rd x).
	\]
	Suppose that each $\mu_{0}^{(n)}$ has an \ac{OM} functional $I_{0}^{(n)} \colon E^{(n)} \to \Reals$. Then the following statements hold:
	\begin{enumerate}[label=(\alph*)]
		\item \label{item:posterior_OM_functional}
		Each $\mu^{(n)}$ has $I^{(n)} \defeq \Phi^{(n)} + I_{0}^{(n)} \colon E^{(n)} \to \Reals$ as an \ac{OM} functional.
		
		\item \label{item:extended_posterior_OM_functional}
		Suppose that property $M(\mu_{0}^{(n)}, E^{(n)})$ holds.
		Then property $M(\mu^{(n)}, E^{(n)})$ also holds, and the global weak modes of $\mu_{0}^{(n)}$ (resp.\ of $\mu^{(n)}$) are the global minimisers of the extended \ac{OM} functional $I_{0}^{(n)} \colon X \to \eReals$ (resp.\ of $I^{(n)} \colon X \to \eReals$).
		
		\item \label{item:Gamma-convergence_of_posteriors}
		Suppose that $I_{0}^{(n)} \xrightarrow{\Gamma} I_{0}^{(\infty)}$ and $\Phi^{(n)} \to \Phi^{(\infty)}$ continuously\footnote{See \Cref{defn:continuous_convergence} for the definition of continuous convergence.
		Note that \Cref{prop:DalMaso_Prop6.20} is agnostic as to which of the two summands converges continuously, and so \Cref{thm:Gamma-convergence_of_posteriors}\ref{item:Gamma-convergence_of_posteriors} also holds if $I_{0}^{(n)} \to I_{0}^{(\infty)}$ continuously and $\Phi^{(n)} \xrightarrow{\Gamma} \Phi^{(\infty)}$, which would be a weaker hypothesis on the potentials \change{but a stronger one on the prior \ac{OM} functionals}.
		However, since we have not studied the continuous convergence of prior \ac{OM} functionals, we do not stress this version of the theorem.} as $n \to \infty$.
		Then the \ac{OM} functionals $I^{(n)}$ satisfy
		\[
			\Gammalim_{n \to \infty} I^{(n)} = I^{(\infty)} .
		\]

		\item \label{item:equicoercivity_of_posteriors}
		Suppose that the sequence $(I_{0}^{(n)})_{n \in \Naturals}$ is equicoercive and the functions $\Phi^{(n)} \geq M$ are uniformly bounded from below by some constant $M \in \Reals$.
		Then the sequence $(I^{(n)})_{n \in \Naturals}$ is also equicoercive with respect to the same representatives of $I^{(n)}$ as for the $\Gamma$-convergence (cf.\ \Cref{remark:UniquenessOMfunctionals}).
		
		\item \label{item:convergence_of_gloabl_weak_modes}
		Suppose that the assumptions of parts \ref{item:extended_posterior_OM_functional}--\ref{item:equicoercivity_of_posteriors} all hold.
		Then the cluster points as $n \to \infty$ of the global weak modes of the posteriors $\mu^{(n)}$ are the global weak modes of the limiting posterior $\mu^{(\infty)}$.
	\end{enumerate}
\end{theorem}

\begin{proof}
	Parts \ref{item:posterior_OM_functional} and \ref{item:extended_posterior_OM_functional} follow from \Cref{lem:reweighted_OM_loc_unif_cts}, and part \ref{item:Gamma-convergence_of_posteriors} follows from \Cref{prop:DalMaso_Prop6.20} \citep[i.e.][Proposition~6.20]{DalMaso1993}.
	
	For part \ref{item:equicoercivity_of_posteriors}, let $(I_{0}^{(n)})_{n \in \Naturals}$ be equicoercive and $\Phi^{(n)} \geq M$ be uniformly bounded from below.
	Then, for any $t \in \Reals$, there exists a compact $K_{t} \subseteq X$ such that, for all $n \in \Naturals$, $(I_{0}^{(n)})^{-1} ([-\infty, t]) \subseteq K_{t}$.
	Since $I^{(n)}(x) = I_{0}^{(n)}(x) + \Phi^{(n)}(x) \leq t$ implies $I_{0}^{(n)}(x) \leq t-M$, it follows that, for any $t \in \Reals$ and $n \in \Naturals$, $(I^{(n)})^{-1} ([-\infty, t]) \subseteq K_{t-M}$.
		
	Finally, part \ref{item:convergence_of_gloabl_weak_modes} is just a restatement of \Cref{thm:convergence_of_modes}.
\end{proof}

\begin{remark}
	Loosely speaking, the hypothesis in \Cref{thm:Gamma-convergence_of_posteriors}\ref{item:equicoercivity_of_posteriors} that the potentials are uniformly bounded below corresponds to a likelihood model in which the observed data are (uniformly) finite dimensional.
	\ac{BIP}s with infinite-dimensional data are known to involve potentials that are unbounded below. Such potentials cannot be interpreted as (non-negative) misfit functionals, as discussed by e.g.\ \citet[Remark~3.8]{Stuart2010} and \citet{KasanickyMandel2017}.

	Note also that a standing assumption of \citet{DashtiLawStuartVoss2013} is that $\Phi$ is locally Lipschitz continuous, which is stronger than the local uniform continuity assumed in \Cref{thm:Gamma-convergence_of_posteriors}, and that boundedness of $\Phi$ from below is also assumed by \citet[Theorem~3.5]{DashtiLawStuartVoss2013}, just as in the hypothesis of \Cref{thm:Gamma-convergence_of_posteriors}\ref{item:equicoercivity_of_posteriors}.
\end{remark}

\change{
\begin{corollary}
	\label{cor:MAP_convergence_under_perturbations}
	Consider a \ac{BIP} with prior $\mu_{0} = \mu_{0}^{(\infty)}$, potential $\Phi = \Phi^{(\infty)}$, and observed data $y = y^{(\infty)}$, each of which may now be approximated.
	In addition to the assumptions of \Cref{thm:Gamma-convergence_of_posteriors}, assume for simplicity that the \ac{OM} functional of $\mu_{0}$ is lower semicontinuous, so that it equals its own $\Gamma$-limit (\Cref{thm:Gamma_limit_of_constant_sequence}).
	\begin{enumerate}[label=(\alph*)]
		\item If the data $y \qefed y^{(\infty)}$ are approximated by a sequence $(y^{(n)})_{n \in \Naturals}$ and the potential $\Phi$ and prior $\mu_{0}$ are held constant, then continuous convergence of $\Phi^{(n)} \defeq \Phi (\quark; y^{(n)})$ to $ \Phi^{(\infty)} \defeq \Phi(\quark; y^{(\infty)})$ ensures $\Gamma$-convergent and equicoercive sequences of posterior \ac{OM} functionals and convergent sequences of \ac{MAP} estimators (in the sense of global weak modes, and up to subsequences).

		\item Similarly, if the data and potential are held constant and the prior $\mu_{0} \qefed \mu_{0}^{(\infty)}$ is approximated by a sequence of priors $( \mu_{0}^{(n)} )_{n \in \Naturals}$, then $\Gamma$-convergence of prior \ac{OM} functionals, i.e.\ $I_{\mu_{0}^{(n)}} \xrightarrow{\Gamma} I_{\mu_{0}^{(\infty)}}$, yields convergent sequences of \ac{MAP} estimators.

		\item Finally, if the data and prior are held constant and the potential $\Phi \qefed \Phi^{(\infty)}$ is approximated by a sequence of potentials $( \Phi^{(n)} )_{n \in \Naturals}$, then continuous convergence $\Phi^{(n)} (\quark; y) \to \Phi^{(\infty)}(\quark; y)$ yields convergent sequences of \ac{MAP} estimators;
		in particular, this holds when the approximate misfit/potential $\Phi^{(n)}$ arises through projection, e.g.\ Galerkin discretisation (\Cref{lem:continuous_convergence_potentials}).
	\end{enumerate}
\end{corollary}
}

\change{
\begin{example}[Small-noise limits]
	Regrettably, the analysis of \ac{MAP} estimators of small-noise (infinite-precision) limits is not entirely trivial even under the $\Gamma$-convergence theory that we have outlined.
	Consider a \ac{BIP} on $X$ with prior $\mu_{0}$ and potential $\Phi$.
	Assume that $\mu_{0}$ has \ac{OM} functional $I_{0} \colon E \to \Reals$ that satisfies $M(\mu_{0}, E)$, leading to a lower semi-continuous and coercive extended \ac{OM} functional $I_{0} \colon X \to \eReals$.
	Assume also that $\Phi$ is locally uniformly continuous, is bounded below, and attains its lower bound --- without loss of generality, take this minimal value to be $0$.
	Now consider the posterior
	\[
		\mu^{(n)} (\rd x) \defeq \frac{1}{Z^{(n)}} e^{- n \Phi (x)} \, \mu_{0} (\rd x)
	\]
	in the small-noise limit $n \to \infty$.
	By \Cref{thm:Gamma-convergence_of_posteriors}, $\mu^{(n)}$ has \ac{OM} functional $I^{(n)} = n \Phi + I_{0}$.
	It is easy to see that, pointwise,
	\[
		\lim_{n \to \infty} I^{(n)} (x)
		=
		I^{(\infty)} (x)
		\defeq
		\begin{cases}
			I_{0} (x) , & \text{if $\Phi(x) = 0$,} \\
			+\infty , & \text{otherwise.}
		\end{cases}
	\]
	It is natural to hope that $\Gammalim_{n \to \infty} I^{(n)} = I^{(\infty)}$ as well, and hence that the \ac{MAP} estimators of $\mu^{(n)}$ converge, in the small-noise limit $n \to \infty$, to the constrained minimisers of the prior \ac{OM} functional $I_{0}$ among the global minima of $\Phi$.
	However, this $\Gamma$-convergence is not straightforward to establish.
	\begin{itemize}
		\item For the $\Gamma$-$\limsup$ inequality, choose any $x \in X$.
		Consider first the case that $\Phi(x) > 0$:
		for the recovery sequence $x_{n} \equiv x$,
		\[
			I^{(n)} (x_{n}) = n \Phi (x) + I_{0} (x) \xrightarrow[n \to \infty]{} + \infty  = I^{(\infty)} (x) .
		\]
		Similarly, in the case $\Phi (x) = 0$, we may use the same recovery sequence to obtain $I^{(n)} (x_{n}) = I^{(n)} (x) = I_{0} (x) = I^{(\infty)} (x)$.
		
		\item For the $\Gamma$-$\liminf$ inequality, choose any $x \in X$ and any sequence $x_{n} \to x$.
		Taking $\omega_{\Phi, x}$ to be a local modulus of continuity for $\Phi$ near $x$, we have
		\[
			\absval{ \Phi (x_{n}) - \Phi (x) } \leq \omega_{\Phi, x} ( \norm{ x_{n} - x } )
		\]
		and hence
		\begin{align*}
			\liminf_{n \to \infty} I^{(n)} (x_{n})
			& = \liminf_{n \to \infty} \bigl( n \Phi (x_{n}) + I_{0} (x_{n}) \bigr) \\
			& \geq \liminf_{n \to \infty} \bigl( n \Phi (x) - n \omega_{\Phi, x} ( \norm{ x_{n} - x } ) + I_{0} (x_{n}) \bigr) \\
			& \geq \lim_{n \to \infty} n \Phi (x) - \liminf_{n \to \infty} n \omega_{\Phi, x} ( \norm{ x_{n} - x } ) + I_{0} (x) ,
		\end{align*}
		where the last inequality uses the lower semicontinuity of $I_{0}$.
		At this point we encounter a problem.
		For $x$ such that $\Phi(x) > 0$, the right-hand side of the above display is indeed $+ \infty$, as required.
		However, for $x$ such that $\Phi (x) = 0$, the $\Gamma$-$\liminf$ inequality only holds if $\liminf_{n \to \infty} n \omega_{\Phi, x} ( \norm{ x_{n} - x } ) = 0$, and this holds only if $x_{n}$ converges sufficiently rapidly to $x$, which is not at all guaranteed.
	\end{itemize}
\end{example}
}

\change{We close this section by repeating the observation made at the end of \Cref{sec:modes_and_OM_functionals}, that the} necessity of the continuous convergence / $\Gamma$-convergence assumptions, as opposed to simple pointwise convergence of densities or \ac{OM} functionals, is \change{shown} by \Cref{example:MAP_counterexamples}\ref{item:pointwise_convergent_density} from the introduction, which can easily be interpreted as a pointwise but not continuously convergent sequence of likelihoods/potentials and a Gaussian prior.

\section{Closing remarks}
\label{sec:closing}

The purpose of this paper was to establish a convergence theory for modes of probability measures (in the \ac{BIP} setting, \ac{MAP} estimators of Bayesian posterior measures) in the sense of maximisers of small ball probabilities, by first characterising them as minimisers of \ac{OM} functionals and then using the well-established notion of $\Gamma$-convergence from the calculus of variations.
The correspondence between modes and \ac{OM} minimisers was established rigorously for global weak modes under the abstract $M$-property, and counterexamples were given to show that an extension to strong modes and relaxation of the $M$-property would be non-trivial if not impossible.
The general programme of studying $\Gamma$-limits of \ac{OM} functionals of measures was illustrated via two explicit example classes that are frequently used in the inverse problems literature, namely Gaussian measures and Besov $B_{p}^{s}$ measures with integrability parameter $p = 1$.

The Gaussian and Besov-1 measures treated in this paper are merely simple examples of a general class of measures, namely countable products of scaled copies of a measure on $\R$ (the normal and Laplace distributions respectively).
General Besov-$p$ measures and infinite-product Cauchy measures fall into this class.
Part~II of this paper \citep{AyanbayevKlebanovLieSullivan2021_II} treats this class in a high degree of generality, following the same programme of determining the \ac{OM} functional, verifying the $M$-property, and showing $\Gamma$-convergence and equicoercivity.
The advantage of having considered the Gaussian and Besov-1 measures separately in this paper is that the requisite calculations could be done in more-or-less closed form and with much less notational overhead than the general case.

This work has made extensive use of the hypothesis that some measure $\mu$ of interest actually possesses an \acl{OM} functional (and moreover one that satisfies property $M (\mu, E)$ for a ``good enough'' $E$), and that $\mu$ possesses a mode.
However, there are examples, even in finite dimension, of $\mu$ that have no strong or global weak modes, only \emph{generalised} modes in the sense of \citet{Clason2019GeneralizedMI}, which are associated with generalised \ac{OM} functionals.
A natural further generalisation of this article would be to study the $\Gamma$-convergence properties of such generalised \ac{OM} functionals, and hence the convergence of generalised strong modes.

\change{It would be of great value in applications not only to know that some sequence of approximations to an ideal limiting \ac{MAP} problem $\Gamma$-converges, but also to quantify how quickly those approximate \ac{MAP} estimators converge.
Unfortunately, this is not trivial, since the basic framework of $\Gamma$-convergence does not easily deliver convergence rates for minimisers, especially when the objective functions are non-smooth, as is the case for most of the \ac{OM} functionals in our setting.
Therefore, the interesting question of convergence rates for modes / \ac{MAP} estimators must be deferred to future work.}

\appendix

\section{\texorpdfstring{$\boldsymbol{\mathsf{\Gamma}}$-convergence}{Gamma-convergence}}
\label{sec:Gamma}

We collect here the basic definitions and results related to $\Gamma$-convergence as used in the main text.
Standard references on $\Gamma$-convergence include the books of \citet{Braides2002, Braides2006} and \citet{DalMaso1993}.

\begin{definition}
	\label{defn:Gamma_and_equicoercive}
	Let $X$ be a metric space and suppose that $F_{n}, F \colon X \to \eReals$.
	We say that $F_{n}$ \defterm{$\boldsymbol{\Gamma}$-converges} to $F$, written $\Gammalim_{n \to \infty} F_{n} = F$ or $F_{n} \xrightarrow[n \to \infty]{\Gamma} F$, if, for every $x \in X$,
	\begin{enumerate}[label=(\alph*)]
		\item ($\Gamma$-$\liminf$ inequality)
		for every sequence $(x_{n})_{n \in \Naturals}$ converging to $x$,
		\begin{equation*}
			F(x) \leq \liminf_{n \to \infty} F_{n} (x_{n}) ;
		\end{equation*}
		\item ($\Gamma$-$\limsup$ inequality)
		and there exists a ``recovery sequence'' $(x_{n})_{n \in \Naturals}$ converging to $x$ such that
		\begin{equation*}
			F(x) \geq \limsup_{n \to \infty} F_{n} (x_{n}) .
		\end{equation*}
	\end{enumerate}
    We say that $(F_{n})_{n\in\Naturals}$ is \defterm{equicoercive} if for all $t \in \Reals$, there exists a compact $K_{t} \subseteq X$ such that, for all $n \in \Naturals$, $F_{n}^{-1} ([-\infty, t]) \subseteq K_{t}$.
\end{definition}

In general, $\Gamma$-convergence and pointwise convergence are independent of one another, although the following inequality always holds:
\begin{equation*}
	\left(
	F_{n} \xrightarrow[n \to \infty]{\Gamma} F \text{ and }
		F_{n} \xrightarrow[n \to \infty]{~} G \text{ pointwise} \right)
	\implies
	F \leq G .
\end{equation*}
However, one can compare $\Gamma$-convergence with continuous convergence:

\begin{definition}
	\label{defn:continuous_convergence}
	Let $X$ be a metric space and suppose that $F_{n}, F \colon X \to \eReals$.
	We say that $F_{n}$ \defterm{converges continuously} to $F$ if, for every $x \in X$ and every neighbourhood $V$ of $F(x)$ in $\eReals$, there exists $N \in \Naturals$ and a neighbourhood $U$ of $x$ such that
	\[
		( n \geq N \text{ and } x' \in U ) \implies F_{n}(x') \in V .
	\]
\end{definition}

Continuous convergence implies both pointwise convergence and $\Gamma$-convergence and, in the case that $F$ is continuous, is implied by uniform convergence of $F_{n}$ to $F$ \citep[Chapters~4 and 5]{DalMaso1993}.

\begin{theorem}[Fundamental theorem of $\mathsf{\Gamma}$-convergence; {\citealp[Theorem 2.10]{Braides2006}}]
	\label{thm:fundamental_Gamma}
	Let $X$ be a metric space and suppose that $F_{n}, F \colon X \to \eReals$ are such that $\Gammalim_{n \to \infty} F_{n} = F$ and $(F_{n})_{n \in \Naturals}$ is equicoercive.
	Then $F$ has a minimum value and $\min_{X} F = \lim_{n \to \infty} \inf_{X} F_{n}$.
	Moreover, if $(x_{n})_{n \in \Naturals}$ is a precompact sequence such that $\lim_{n \to \infty} F_{n} (x_{n}) = \min_{X} F$, then every limit of a \change{convergent} subsequence of $(x_{n})_{n \in \Naturals}$ is a minimiser of $F$.
	Thus, if each $F_{n}$ has a minimiser $x_{n}$, then every convergent subsequence of $(x_{n})_{n\in\Naturals}$ has as its limit a minimiser of $F$.
\end{theorem}

\begin{proposition}[{\citealp[Proposition~6.20]{DalMaso1993}}]%
	\label{prop:DalMaso_Prop6.20}%
	Let $X$ be a metric space and suppose that $F_{n}, F \colon X \to \eReals$ and $G_{n}, G \colon X \to \Reals$ are such that $F_{n} \xrightarrow{\Gamma} F$ on $X$ and $G_{n} \to G$ continuously on $X$ as $n \to \infty$.
	Then
	\[
		F_{n} + G_{n} \xrightarrow[n \to \infty]{\Gamma} F + G
		\quad
		\text{on $X$.}
	\]
\end{proposition}

\begin{theorem}[{\citealp[Proposition~2.5]{Braides2006}}]%
	\label{thm:Gamma_limit_of_constant_sequence}%
	The $\Gamma$-limit of a constant sequence $(F)_{n \in \Naturals}$ is the lower semicontinuous envelope $F^{\text{lsc}}$ of $F$, i.e.\ the greatest lower semicontinuous function bounded above by $F$:
	\begin{equation*}%
		F^{\text{lsc}} (x) \defeq \liminf_{x' \to x} F(x') .
	\end{equation*}
	In particular, $F = \Gammalim_{n \to \infty} F$ if and only if $F$ is lower semicontinuous.
\end{theorem}

\section{Technical supporting results}
\label{sec:technical}

\subsection{Supporting results for \Cref{sec:notation}}

\begin{lemma}[The $M$-property]
	\label{lem:vanishing_small_ball_prob}
	Let $X$ be a metric space and let $\mu_{0} \in \prob{X}$.
	Suppose that $\mu_{0}$ has an \ac{OM} functional $I \colon E \to \Reals$ on a nonempty subset $E\subseteq \supp(\mu)$.
	\begin{enumerate}[label=(\alph*)]
		\item
		\label{item:vanishing_small_ball_prob_xstar_or_xstarstar}
		If $x^{\star} \in E$ satisfies \eqref{eq:liminf_small_ball_prob}, then any $x^{\star\star} \in E\setminus\{x^{\star}\}$ satisfies \eqref{eq:liminf_small_ball_prob}.
		In particular, property $M (\mu_{0},E)$ does not depend on the choice of $x^{\star}$ in \eqref{eq:liminf_small_ball_prob}.
		\item
		\label{item:vanishing_small_ball_prob_transfer}
		Let $\Phi \colon X \to \Reals$ be measurable, such that $\Phi$ is bounded on bounded subsets of $X$.\footnote{%
		The assumption that $\Phi$ is bounded on bounded subsets of $X$ is not restrictive.
		If $\Phi \colon X\to\Reals$ is continuous, then the boundedness assumption holds whenever the bounded subset is contained in a sufficiently small ball.
		Mild continuity assumptions on forward models and log-likelihoods are commonplace in the study of \ac{BIP}s \citep[e.g.][Assumption~2.6(i--ii)]{Stuart2010}.%
		}
		Suppose that $Z \defeq \int_{X} e^{-\Phi(x)} \, \mu_{0} (\rd x) \in (0, \infty)$ and define $\mu \in \prob{X}$ by $\mu (\rd x) \defeq Z^{-1} e^{-\Phi(x)} \, \mu_{0} (\rd x)$.
		If property $M (\mu_{0}, E)$ holds, then so too does $M (\mu, E)$.
		\item
		\label{item:vanishing_small_ball_prob_not_weak_mode}
		If property $M (\mu_{0}, E)$ holds, then no point of $X \setminus E$ can be a global weak mode for $\mu_{0}$, and hence cannot be a strong mode for $\mu_{0}$.
	\end{enumerate}
\end{lemma}

\begin{proof}
	Suppose that $x^{\star}\in E$ satisfies \eqref{eq:liminf_small_ball_prob}.
	Let $x^{\star\star} \in E$ and $x \in X \setminus E$.
	If $x\notin\supp(\mu_{0})$, then for sufficiently small $r$, $\mu_0( \cBall{x}{r})=0$, so it suffices to prove the claim for $x \in \supp(\mu_{0})\setminus E$.
	Since $\mu_{0}$ has an \ac{OM} functional $I \colon E\to\Reals$,
	\begin{align*}
		\lim_{r \searrow 0} \frac{ \mu_{0} ( \cBall{x}{r} ) }{ \mu_{0} ( \cBall{x^{\star\star}}{r} ) }
		& = \lim_{r \searrow 0} \frac{ \mu_{0} ( \cBall{x^{\star}}{r} ) }{ \mu_{0} ( \cBall{x^{\star\star}}{r} ) } \frac{ \mu_{0} ( \cBall{x}{r} ) }{ \mu_{0} ( \cBall{x^{\star}}{r} ) } \\
		& = \lim_{r \searrow 0} \frac{ \mu_{0} ( \cBall{x^{\star}}{r} ) }{ \mu_{0} ( \cBall{x^{\star\star}}{r} ) } \lim_{r \searrow 0} \frac{ \mu_{0} ( \cBall{x}{r} ) }{ \mu_{0} ( \cBall{x^{\star}}{r} ) } \\
		& = \exp ( I(x^{\star\star}) - I(x^{\star}) ) \lim_{r \searrow 0} \frac{ \mu_{0} ( \cBall{x}{r} ) }{ \mu_{0} ( \cBall{x^{\star}}{r} ) } \\
		& = 0 ,
	\end{align*}
	where we used \eqref{eq:Onsager--Machlup} and \eqref{eq:liminf_small_ball_prob} in the penultimate and last equation.
	This proves \ref{item:vanishing_small_ball_prob_xstar_or_xstarstar}.

	For \ref{item:vanishing_small_ball_prob_transfer}, suppose that property $M (\mu_{0}, E)$ holds \change{with $x^{\star} \in E$ satisfying
	\[
	x \in X \setminus E \implies \lim_{r \searrow 0} \frac{ \mu ( \cBall{x}{r} ) }{ \mu ( \cBall{x^{\star}}{r} ) } = 0.
	\]
	}
	Observe that, for any $x \in X$,
	\begin{align*}
		\frac{ \mu ( \cBall{x}{r} ) }{ \mu ( \cBall{x^{\star}}{r} ) }
		= \frac{ \int_{\cBall{x}{r}} \exp ( - \Phi(y) ) \, \mu_{0} ( \rd y ) }{ \int_{\cBall{x^{\star}}{r}} \exp ( - \Phi(y) ) \, \mu_{0} ( \rd y ) }
		\leq \exp \left( - \inf_{\cBall{x}{r}} \Phi + \sup_{\cBall{x^{\star}}{r}} \Phi \right) \frac{ \mu_{0} ( \cBall{x}{r} ) }{ \mu_{0} ( \cBall{x^{\star}}{r} ) } .
	\end{align*}
	The exponential on the right-hand side is finite, by the assumption that $\Phi$ is bounded on bounded subsets of $X$.
	If $x \in X \setminus E$, then taking the limit as $r \searrow 0$ yields property $M (\mu, E)$, as claimed.

	Suppose that $x \in X \setminus E$ is a global weak mode in the sense of \cref{defn:global_weak_mode}.
	Then $x \in \supp (\mu)$ and
	\begin{equation*}
	 1\geq \limsup_{r \searrow 0} \frac{ \mu_{0} ( \cBall{x^{\star}}{r} ) }{ \mu_{0} ( \cBall{x}{r} ) }=\left( \liminf_{r \searrow 0} \frac{ \mu_{0} ( \cBall{x}{r} ) }{ \mu_{0} ( \cBall{x^{\star}}{r} ) }\right)^{-1}.
	\end{equation*}
	Above, we used that $x^{\star}\in E \subseteq \supp (\mu)$ to ensure that for every $r>0$, $\tfrac{ \mu_{0} ( \cBall{x^{\star}}{r} ) }{ \mu_{0} ( \cBall{x}{r} ) }>0$.
	The inequality above implies that
    \begin{equation*}
      \liminf_{r \searrow 0} \frac{ \mu_{0} ( \cBall{x}{r} ) }{ \mu_{0} ( \cBall{x^{\star}}{r} ) }\geq 1
    \end{equation*}
    and hence $x^{\star}$ does not satisfy \eqref{eq:liminf_small_ball_prob}.
    Finally, if $x$ is not a global weak mode, then by \cref{lem:strong_mode_implies_global_weak_mode}, it cannot be a strong mode.
    This proves \ref{item:vanishing_small_ball_prob_not_weak_mode}.
\end{proof}

\begin{figure}
	\centering
	\includegraphics[width=0.9\textwidth]{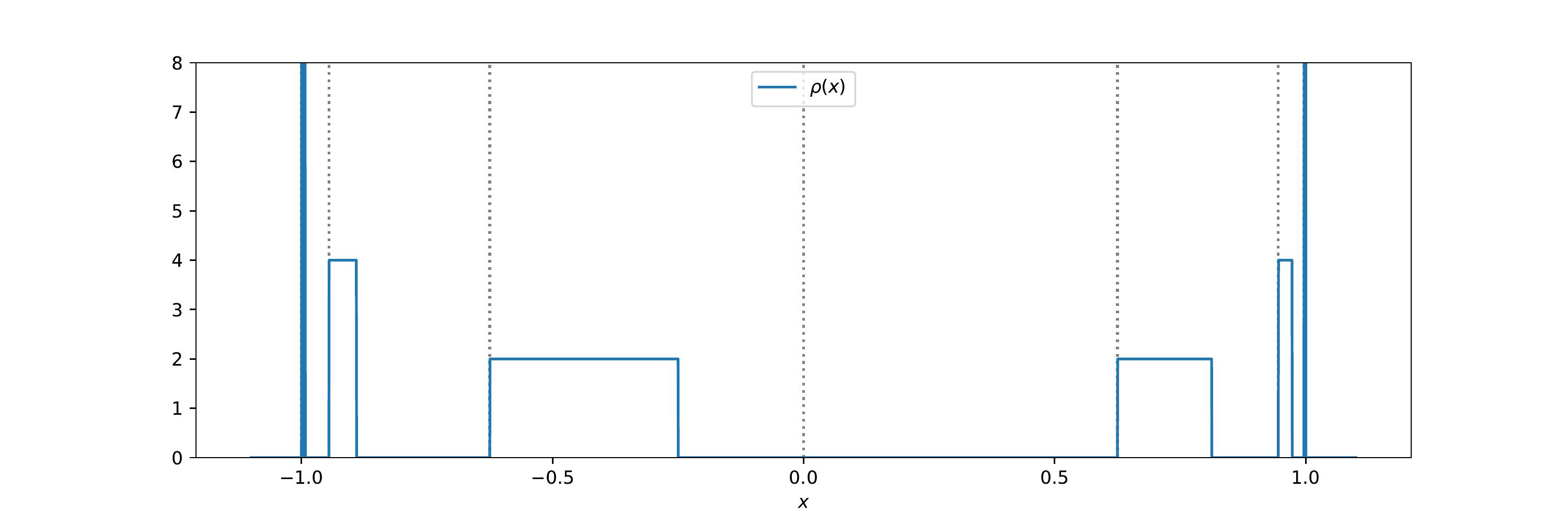}
	\caption{Probability density $\rho$ (unnormalized) in \Cref{example:Minimiser_of_OM_not_weak_mode}.
		$u=1$ is a minimiser of the \ac{OM} functional $I_{\mu,E}$ on $E = \{1\}$, but not a global weak mode, even though the ``$\liminf$-only'' version of property $M(\mu,E)$ is satisfied.
		The vertical lines are located at $-1+\alpha_{n}$ and $1-\alpha_{n}$, respectively.
	}
	\label{fig:Minimiser_of_OM_not_weak_mode}
\end{figure}

\begin{example}[Importance of limit in the $M$-property]
	\label{example:Minimiser_of_OM_not_weak_mode}
	For $n\in\Naturals$, let
	\begin{align*}
		a_{n} & \defeq 2^{-\frac{(n-1)(n+2)}{2}},
		&
		b_{n} & \defeq \frac{a_{n}}{2},
		&
		\alpha_{n} & \defeq 2^{-n} (a_{n} - a_{n+1}),
		&
		\beta_{n} & \defeq \frac{\alpha_{n}}{2},
		&
		\varepsilon_{n} & \defeq  2\alpha_{n},
		&
		\delta_{n} & \defeq \alpha_{n}.
	\end{align*}
	Let $\mu \in \prob{\R}$ be given by its probability density
	\[
		\rho
		\propto
		\sum_{n \in \Naturals} 2^{n} \big( \one_{[-1+\alpha_{n},-1+2\alpha_{n}]} + \one_{[1-2\beta_{n},1-\beta_{n}]}\big),
	\]
	where $\one_{A}$ denotes the indicator function of $A\subseteq \Reals$, as visualized in \Cref{fig:Minimiser_of_OM_not_weak_mode}.
	The definitions of $a_{n}$ and $\alpha_{n}$ imply that $(a_{n})_{n\in\Naturals}$ is strictly decreasing to $0$ and, for every $n\in\Naturals$,
	\begin{equation*}
		\alpha_{n}= 2^{-\frac{n^{2}+3n-2}{2}}-2^{-\frac{n^{2}+5n}{2}}>0,\quad  \frac{\alpha_{n}}{\alpha_{n+1}}>\frac{ 2^{-\frac{n^{2}+3n-2}{2}}-2^{-\frac{n^{2}+5n}{2}}}{2^{-\frac{n^2+5n+2}{2}}}=2^{n+2}-2^{-1}>2,
	\end{equation*}
	so $(\alpha_{n})_{n\in\Naturals}$ is strictly positive and strictly decreasing to $0$.
	Hence, the intervals $([-1+\alpha_{n},-1+2\alpha_{n}])_{n\in\Naturals}$ are disjoint, and the intervals $([1-2\beta_{n},1-\beta_{n}])_{n\in\Naturals}$ are also disjoint.
	The facts that $(\alpha_{n})_{n\in\Naturals}$ are strictly positive and decreasing imply that if $m\geq n$, then $(-1+\alpha_{m},-1+2\alpha_{m})\subseteq (-1-2\alpha_{n},-1+2\alpha_{n})=\cBall{-1}{\varepsilon_{n}}$.
	The inequality $\tfrac{\alpha_{n-1}}{\alpha_{n}}>2$ implies that $-1+2\alpha_{n}<-1+\alpha_{m}$ and hence $(-1+\alpha_{m},-1+2\alpha_{m})\cap (-1-2\alpha_{n},-1+2\alpha_{n})=\emptyset$ for any $m<n$.
	This implies that $\mu(\cBall{-1}{\varepsilon_{n}})= \sum_{k=n}^{\infty}2^{k} \alpha_{k}=a_{n}$.
	Similar arguments yield
	\begin{align}
		\label{equ:Large_lim_sup}
		\frac{\mu(\cBall{-1}{\varepsilon_{n}})}{\mu(\cBall{1}{\varepsilon_{n}})}
		&=
		\frac{\sum_{k=n}^{\infty} 2^{k} \alpha_{k}}{\sum_{k=n}^{\infty} 2^{k} \beta_{k}}
		=
		\frac{a_{n}}{b_{n}}
		=
		2,
		\\
		\label{equ:Lim_inf_equals_zero}
		\frac{\mu(\cBall{-1}{\delta_{n}})}{\mu(\cBall{1}{\delta_{n}})}
		&=
		\frac{\sum_{k=n+1}^{\infty} 2^{k} \alpha_{k}}{\sum_{k=n}^{\infty} 2^{k} \beta_{k}}
		=
		\frac{a_{n+1}}{b_{n}}
		=
		2^{-n-2}
		\xrightarrow[n\to\infty]{}
		0.
	\end{align}
	Thus, $\limsup_{\varepsilon \searrow 0} \frac{\mu(\cBall{-1}{\varepsilon})}{\mu(\cBall{1}{\varepsilon})} = 2$ and $\liminf_{\varepsilon \searrow 0} \frac{\mu(\cBall{-1}{\varepsilon})}{\mu(\cBall{1}{\varepsilon})} = 0$.
	For $x\in \supp(\mu)\setminus\{-1,1\}$, it follows from the disjointness of $([-1+\alpha_{n},-1+2\alpha_{n}])_{n\in\Naturals}$ and $([1-2\beta_{n},1-\beta_{n}])_{n\in\Naturals}$ that there exists a unique $m\in\Naturals$ such that either $x\in [-1+\alpha_{m},-1+2\alpha_{m}]$, or $x\in [1-2\beta_{m},1-\beta_{m}]$.
	In either case, for sufficiently small $\varepsilon$ it holds that $\mu(\cBall{x}{\varepsilon})=(2\varepsilon )2^m$, if $x$ is in the interior of either interval.
	If $x$ is an endpoint of one of the intervals, then for sufficiently small $\varepsilon$, $\mu(\cBall{x}{\varepsilon})=\varepsilon 2^m$, so it suffices to consider the case where $x$ is in the interior of one of the intervals.
	Since $\mu(\cBall{1}{\varepsilon_{n}})=b_{n}$, it follows that
	\begin{equation*}
		\frac{\mu(\cBall{x}{\varepsilon_{n}})}{\mu(\cBall{1}{\varepsilon_{n}})} = \frac{2^{m+1}2\alpha_{n}}{b_n}=\frac{2^{m+1-n}(a_{n}-a_{n+1})}{a_n}<2^{m+1-n}\frac{a_{n}}{a_{n}}=2^{m+1-n}\xrightarrow[n\to\infty]{}
		0.
	\end{equation*}
	Then $\liminf_{\varepsilon \searrow 0} \frac{\mu(\cBall{x}{\varepsilon})}{\mu(\cBall{1}{\varepsilon})} = 0$ for any point.
	Hence, for $E = \{ 1 \}$, the $\liminf$ part of property $M(\mu,E)$ is satisfied.
	Further, $u=1$ is a minimiser of any \ac{OM} functional on $E$ since $E=\{1\}$, but $u=1$ is not a global weak mode due to \eqref{equ:Large_lim_sup}.
	
	Note that the above example can be modified to be even more extreme:
	if one sets $a_{n} \defeq 2^{-n(n-1)}$ and $b_{n} \defeq \frac{a_{n}}{2^{n}}$, then one obtains that $a_{n}/b_{n} = 2^{n}$ and $a_{n+1}/b_{n} = 2^{-n}$, and hence
	\[
		\liminf_{\varepsilon \searrow 0} \frac{\mu(\cBall{-1}{\varepsilon})}{\mu(\cBall{1}{\varepsilon})}
		=
		0,
		\qquad
		\limsup_{\varepsilon \searrow 0} \frac{\mu(\cBall{-1}{\varepsilon})}{\mu(\cBall{1}{\varepsilon})}
		=
		\infty.
	\]
\end{example}

\begin{proposition}[Open v.\ closed balls]
	\label{prop:Open_Closed_Coincide}
	Let $X$ be a metric space, $\mu \in \prob{X}$ a probability measure on $(X,\Borel{X})$ and $x_{1},\, x_{2}\in X$ with $x_{2} \in \supp(\mu)$.
	For $\varepsilon>0$ define the ratios $\oRatio{\varepsilon} \defeq \frac{ \mu ( \cBall{x_{1}}{\varepsilon} ) }{ \mu ( \cBall{x_{2}}{\varepsilon} ) }$ and $\cRatio{\varepsilon} \defeq \frac{ \mu ( \ccBall{x_{1}}{\varepsilon} ) }{ \mu ( \ccBall{x_{2}}{\varepsilon} ) }$, where $\ccBall{x}{r}$ denotes the closed ball in $X$ of radius $r$ centred on $x$.
	Then
	\[
		\limsup_{\varepsilon \searrow 0} \cRatio{\varepsilon}
		=
		\limsup_{\varepsilon \searrow 0} \oRatio{\varepsilon}
		\qquad
		\text{and}
		\qquad
		\liminf_{\varepsilon \searrow 0} \cRatio{\varepsilon}
		=
		\liminf_{\varepsilon \searrow 0} \oRatio{\varepsilon}.
	\]
	Hence, $\lim_{\varepsilon \searrow 0} \cRatio{\varepsilon}$ exists if and only if $\lim_{\varepsilon \searrow 0} \oRatio{\varepsilon}$ exists, in which case these two values agree.
\end{proposition}

\begin{proof}
	First assume that $\limsup_{\varepsilon \searrow 0} \cRatio{\varepsilon} > \limsup_{\varepsilon \searrow 0} \oRatio{\varepsilon} \eqqcolon \mathring{s}$.
	Then there exists $\zeta > 0$ and a positive null sequence $(\varepsilon_{n})_{n\in\Naturals}$ such that $\cRatio{\varepsilon_{n}} \geq \mathring{s} + \zeta$.
	For each $n\in\Naturals$ perform the following construction:
	Since $\bigcap_{\delta > 0} \cBall{x}{\varepsilon_{n} + \delta} = \ccBall{x}{\varepsilon_{n}}$ for any $x\in X$ and using continuity of probability measures, we obtain
	\[
		\lim_{\delta \searrow 0} \oRatio{\varepsilon_{n} + \delta}
		=
		\frac{ \lim_{\delta \searrow 0} \mu ( \cBall{x_{1}}{\varepsilon_{n} + \delta} ) }{ \lim_{\delta \searrow 0} \mu ( \cBall{x_{2}}{\varepsilon_{n} + \delta} ) }
		=
		\frac{ \mu ( \ccBall{x_{1}}{\varepsilon_{n}} ) }{ \mu ( \ccBall{x_{2}}{\varepsilon_{n}} ) }
		=
		\cRatio{\varepsilon_{n}}
		\geq
		\mathring{s} + \zeta
	\]
	and there exits $0 < \delta_{n} < n^{-1}$ such that $\oRatio{\varepsilon_{n} + \delta_{n}} \geq \mathring{s} + \zeta/2$.
	Hence, we have constructed a null sequence $(\tilde{\varepsilon}_{n})_{n\in\Naturals} \defeq (\varepsilon_{n} + \delta_{n})_{n\in\Naturals}$ with
	\[
		\mathring{s}
		=
		\limsup_{\varepsilon \searrow 0} \oRatio{\varepsilon}
		\geq
		\limsup_{n\to \infty} \oRatio{\tilde{\varepsilon}_{n}}
		\geq \mathring{s} + \zeta/2,
	\]
	which is a contradiction.
	Therefore, our assumption was false and $\limsup_{\varepsilon \searrow 0} \cRatio{\varepsilon} \leq \limsup_{\varepsilon \searrow 0} \oRatio{\varepsilon}$.
	The other inequality can be proven similarly using $\bigcup_{\delta > 0} \ccBall{x}{\varepsilon_{n} - \delta} = \cBall{x}{\varepsilon_{n}}$ and a similar argument works for the corresponding $\liminf$ statement.
\end{proof}

\begin{example}[\ac{OM} functionals and changes of metric]%
	\label{example:OM_and_change_of_metric}%
	Following \citet[Example~5.6]{LieSullivan2018}, let $\mu$ be the finite Borel measure on $(\Reals^{2}, \Borel{\Reals^{2}})$ that is one-dimensional Hausdorff measure (i.e.\ uniform length measure) on the disjoint union $E$ of two right-angled crosses in the plane, with one cross, $E_{+}$, aligned with the coordinate axes and centred at $e_{1} \defeq (1, 0)$ and the other, $E_{-}$, aligned at $\pi / 4$ to the axes and centred at $-e_{1}$, \change{as illustrated in \Cref{figure:OM_and_change_of_metric}.}
	(Note that there is a slight error in \citep[Example~5.6]{LieSullivan2018} concerning the side lengths of the cross $E_{-}$ and hence the total mass of $\mu$, but this error does not affect the final conclusion of that example or this one, since it is only the mass near $\pm e_{1}$ that is important.)
	With respect to the $1$-norm,
	\begin{align*}
		\mu( \cBallExtra{-e_{1}}{r}{1} ) & = 2 \sqrt{2} r , &
		\mu( \cBallExtra{e_{1}}{r}{1} ) & = 4 r , \\
		\intertext{whereas, with respect to the $\infty$-norm, which in this setting is Lipschitz equivalent to the $1$-norm,}
		\mu( \cBallExtra{-e_{1}}{r}{\infty} ) & = 4 \sqrt{2} r , &
		\mu( \cBallExtra{e_{1}}{r}{\infty} ) & = 4 r ,
	\end{align*}
	and, after considering the other points of $\Reals^{2}$, it follows that $e_{1}$ (resp.\ $-e_{1}$) is the unique strong and global weak mode of $\mu$ with respect to the $1$-norm (resp.\ $\infty$-norm).
	These same calculations, though, \change{can be used to} show that $\mu$ has an \ac{OM} functional \change{$I^{1} \colon E \to \Reals$ (resp.\ $I^{\infty} \colon E \to \Reals$) with respect to the $1$-norm (resp.\ $\infty$-norm) and moreover}
	\begin{align*}
		I^{1}(-e_{1}) & = I^{1}(e_{1}) + \log \sqrt{2} , \\
		I^{\infty}(-e_{1}) & = I^{\infty}(e_{1}) - \log \sqrt{2} .
	\end{align*}
	Indeed, more generally, $I^{1}$ takes greater values on $E_{-}$ than on $E_{+}$, whereas $I^{\infty}$ takes greater values on $E_{+}$ than on $E_{-}$.
	This shows that these two \ac{OM} functionals are distinct (i.e.\ differ by more than an additive constant).
\end{example}

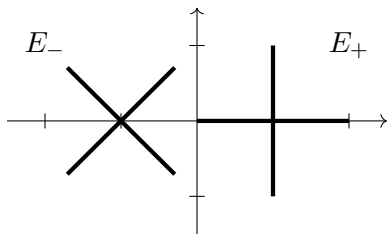
\begin{figure}[t]
	\centering
	\begin{tikzpicture}
		\draw[->] (-2.5, 0.0) -- (2.5, 0.0) ;
		\draw (-2.0, -0.1) -- (-2.0, 0.1) ;
		\draw (-1.0, -0.1) -- (-1.0, 0.1) ;
		\draw (1.0, -0.1) -- (1.0, 0.1) ;
		\draw (2.0, -0.1) -- (2.0, 0.1) ;
		\draw[->] (0.0, -1.5) -- (0.0, 1.5) ;
		\draw (-0.1, 1.0) -- (0.1, 1.0) ;
		\draw (-0.1, -1.0) -- (0.1, -1.0) ;
		\draw[ultra thick] (0.0, 0.0) -- (2.0, 0.0) ;
		\draw[ultra thick] (1.0, -1.0) -- (1.0, 1.0) ;
		\draw[ultra thick] (-1.707, 0.707) -- (-0.293, -0.707) ;
		\draw[ultra thick] (-1.707, -0.707) -- (-0.293, 0.707) ;
		\node at (-2, 1) {$E_{-}$} ;
		\node at (2, 1) {$E_{+}$} ;
	\end{tikzpicture}
	\caption{\change{Uniform length measure $\mu$ on the set $E = E_{-} \cup E_{+} \subset \Reals^{2}$ shown here has a unique strong mode and global weak mode at $(1, 0)$ with respect to the $1$-norm on $\Reals^{2}$, but at $({-1}, 0)$ with respect to the $\infty$-norm.
	The \ac{OM} functionals for $\mu$ associated to these two norms are likewise distinct, as discussed in \Cref{example:OM_and_change_of_metric}.}}
	\label{figure:OM_and_change_of_metric}
\end{figure}

\subsection{Supporting results for \Cref{sec:modes_and_OM_functionals}}

\begin{example}
	\label{example:Minimiser_of_OM_not_strong_mode}
	Let $X = \Reals$, $E = \Naturals$ and $\mu \in \prob{X}$ have Lebesgue density \change{$\rho \defeq \frac{24}{5\pi^{2}} \sum_{k\in\Naturals} \rho_{k}$}, shown in \Cref{fig:Minimiser_of_OM_not_strong_mode}, where
	\begin{align}
		\rho_{0}(x)
		& \defeq
		\tfrac{1}{4} \, (\absval{x}^{-1/2} - 2) \, \one_{[-\frac{1}{4},\frac{1}{4}]\setminus\{0\}}(x) \\
		\rho_{k}(x)
		& \defeq
		\change{\frac{ \rho_{0}(x-k) }{ k^{2} } + k^{2}\, \one_{\big[-\frac{1}{2k^{4}},\frac{1}{2k^{4}}\big]}(x -k) .}
	\end{align}
	Since $\supp \rho_0=[-\tfrac{1}{4},\tfrac{1}{4}]$ and \change{$\supp \rho_{k}=[k-\tfrac{1}{4},k+\tfrac{1}{4}]$,} it follows that if \change{$k\neq \ell $ then $\supp \rho_{k}\cap \supp \rho_{\ell}=\emptyset$.}
	In addition, for $r\leq \tfrac{1}{4}$, $\int_{\cBall{0}{r}} \rho_0(x)\ud x =r^{1/2}-r$.
	This implies that $\int_{X} \rho_0=\tfrac{1}{4}$.
	Hence \change{$\int_{X} \rho_{k} = \frac{5}{4k^{2}}$} and $\int_{X}\rho =1$.
	For every \change{$k\in\Naturals$ and $r\leq \min\{\tfrac{1}{4},\tfrac{1}{2k^4}\}$, $\int_{\cBall{k}{r}}\rho_{k}(x)\ud x=k^{-2} (r^{1/2}-r)+2rk^2$.}
	Then, for $u = 1$ and any \change{$k\in E = \Naturals$,
	\[
		\frac{\mu(\cBall{u}{r})}{\mu(\cBall{k}{r})}
		\xrightarrow{r \searrow 0}
		k^{2},
	\]}
	which implies that $u=1$ is \change{a weak mode or weak MAP estimate in the sense of \citep[Definition 4]{HelinBurger2015} and hence a global weak mode in the sense of \Cref{defn:global_weak_mode}. In addition, $I_{\mu,E}(k) = 2 \log k$ defines} an \ac{OM} functional for $\mu$ on $E$, with unique minimiser $u=1$.
	Next, we show that property $M(\mu,E)$ holds.
	If \change{$x\notin \bigcup_{k\in\Naturals}\supp\rho_{k}$,} then for sufficiently small $r>0$, $\mu(\cBall{x}{r})=0$.
	Thus, it suffices to consider \change{$x\in\bigcup_{k\in\Naturals}\supp\rho_{k}\setminus\Naturals$.}
	There exists a unique \change{$m\in\Naturals=E$} such that $x\in \supp\rho_m\setminus\{m\}$.
	Assume that $x=m+\delta$ for $0<\delta\leq\tfrac{1}{4}$.
	Then, for $r$ small enough,
	\begin{equation*}
		\mu(\cBall{x}{r})=\frac{1}{2m^2}\left(\sqrt{\delta+r}-\sqrt{\delta-r}\right)-r+2m^2r\approx \frac{1}{2m^2}\left(\frac{1}{2\delta^{1/2}} 2r\right)+r(2m^2-1),
	\end{equation*}
    using the Taylor expansion of $y\mapsto\sqrt{y}$.
    Thus, $\mu(\cBall{x}{r})$ decreases to zero linearly in $r$, whereas \change{$\mu(\cBall{m}{r})$} decreases to zero like $r^{1/2}$.
    \change{Recall that property $M(\mu,E)$ holds if there exists some $x^{\star}\in E$ such that if $x\in X \setminus E$ then \eqref{eq:liminf_small_ball_prob} holds, i.e. $\lim_{r \searrow 0} \frac{ \mu ( \cBall{x}{r} ) }{ \mu ( \cBall{x^{\star}}{r} ) } = 0$.
    Using that $x=m+\delta$ and $x^{\star}=m$ shows that property $M(\mu,E)$ holds.}
    
    \change{Next, recall from \eqref{eq:Onsager--Machlup} that if $I=I_{\mu,E}:E\to\Reals$ is an \ac{OM} functional for $\mu$, then $I$ must satisfy $\lim_{r \searrow 0} \frac{ \mu ( \cBall{x_{1}}{r} ) }{ \mu ( \cBall{x_{2}}{r} ) }
		=
		\exp ( I(x_{2}) - I(x_{1}) )$ for all $x_{1}, x_{2} \in E$. 
    For $x_1=1$ and $x_2=\tfrac{5}{4}$, the preceding calculations show that the ratio $\frac{ \mu ( \cBall{x_{1}}{r} ) }{ \mu ( \cBall{x_{2}}{r} ) }$ increases to $\infty$ as $r\searrow 0$.}
    If, on the other hand, we set $x_1=\tfrac{5}{4}$ and $x_2=1$, then the resulting limiting ratio equals $0$.
    This shows that the domain of the \ac{OM} functional $I_{\mu,E}$ cannot be extended beyond $E=\Naturals$.
	However, for $u = 1$, $n \in E = \Naturals$ and $r_{n} = \frac{1}{2n^{4}}$,
	\[
		\liminf_{r \searrow 0}
		\frac{\mu(\cBall{u}{r})}{M_{r}}
		\leq
		\liminf_{n \to \infty}
		\frac{\mu(\cBall{u}{r_{n}})}{\mu(\cBall{n}{r_{n}})}
		\leq
		\liminf_{n \to \infty}
		\frac{\frac{1}{\sqrt{2}\, n^{2}} + \frac{1}{n^{4}}}{\frac{1}{n^{2}}}
		=
		\frac{1}{\sqrt{2}}.
	\]
	Thus, $u=1$ cannot be a strong mode of $\mu$, even though it minimises $I_{\mu,E}$.
\end{example}

\begin{figure}
	\centering
	\includegraphics[width=0.9\textwidth]{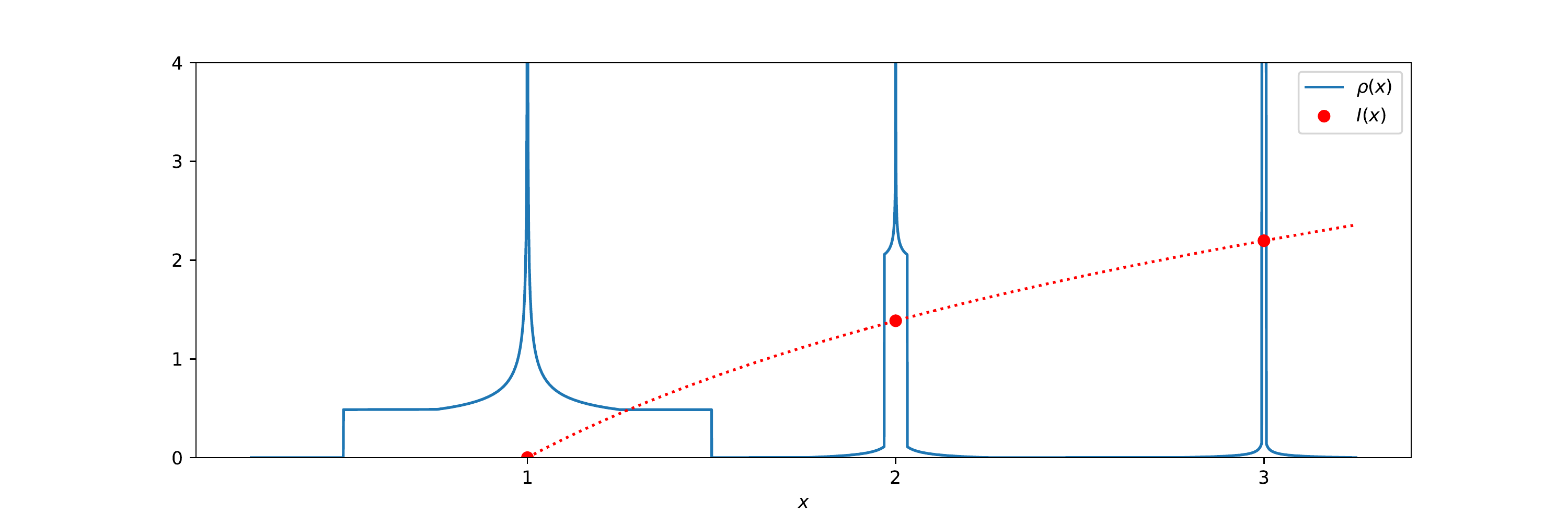}
	\caption{Probability density $\rho$ in \Cref{example:Minimiser_of_OM_not_strong_mode}.
	$u=1$ is a minimiser of the \ac{OM} functional $I_{\mu,E}$ and an $E$-weak mode for $E = \Naturals$, but not a strong mode.
	}
	\label{fig:Minimiser_of_OM_not_strong_mode}
\end{figure}

\subsection{Supporting results for \Cref{sec:convergence_of_OM_functionals}}

\begin{lemma}
	\label{lemma:TechnicalLemmaOnConvergenceInl2}
	Let $a^{(n)} = (a_{k}^{n})_{k\in\N} \in \ell^{2}$, $n\in\N$, define a bounded sequence in $\ell^{2}$, i.e.\ there exists a constant $M > 0$ such that $\norm{a^{(n)}}_{\ell^{2}} \leq M$ for each $n\in\N$.
	Further, let $a_{k}^{(n)} \xrightarrow[n\to\infty]{} a_{k} \in \R$ for each $k\in \N$.
	Then $a \defeq (a_{k})_{k\in\N} \in \ell^{2}$ and $\norm{a}_{\ell^{2}} \leq M$.
\end{lemma}

\begin{proof}
	Assume that there exists $K \in \N$ such that $\sum_{k=1}^{K} \absval{a_{k}}^{2} > M^{2} + \varepsilon$ for some $0 < \varepsilon < 1$.
	Since $a_{k}^{(n)} \xrightarrow[n\to\infty]{} a_{k}$ for each $k\in \N$, there exists, for each $k = 1,\dots,K$, a number $N(k)\in\N$ such that, for all $n\geq N(k)$, $\absval{a_{k}^{(n)} - a_{k}} < \frac{\varepsilon}{4K(\absval{a_{k}} + 1)} \ (\leq 1)$.
	Hence, for each $k = 1,\dots,K$ and $N \defeq \max(N(1),\dots,N(K))$,
	\[
		\absval{a_{k}}^{2}
		\leq
		(\absval{a_{k}^{(N)}} + \absval{a_{k}^{(N)} - a_{k}})^{2}
		=
		\absval{a_{k}^{(N)}}^{2}
		+ 2 \underbrace{\absval{a_{k}^{(N)}}}_{\leq \absval{a_{k}}+1}
		\underbrace{\absval{a_{k}^{(N)} - a_{k}}}_{\leq \frac{\varepsilon}{4K(\absval{a_{k}} + 1)}}
		+ \underbrace{\absval{a_{k}^{(N)} - a_{k}}^{2}}_{\leq \frac{\varepsilon}{4K(\absval{a_{k}} + 1)}}
		\leq
		\absval{a_{k}^{(N)}}^{2} + \frac{\varepsilon}{K}.
	\]
	Therefore, $\sum_{k=1}^{K} \absval{a_{k}}^{2} \leq M^{2} + \varepsilon$, yielding a contradiction.
	Hence, our assumption was false and the lemma is proven.
\end{proof}

\begin{remark}
	\Cref{lemma:TechnicalLemmaOnConvergenceInl2} does not state that $\norm{a^{(n)} - a}_{\ell^{2}} \to 0$ and, in fact, this is not true in general.
	A counterexample is provided by $a=0$ and $a^{(n)} = (\delta_{nk})_{k\in\N}$, where $\delta_{nk}$ denotes the Kronecker delta function.
\end{remark}

\subsection{Supporting results for \Cref{sec:MAP_in_BIP}}

Recall that a function $f \colon X \to Y$ between metric spaces $X$ and $Y$ is \defterm{locally uniformly continuous} if, for every $x \in X$, there exists a function $\omega_{f, x} \colon [0, \infty) \to [0, \infty]$, a \defterm{local modulus of continuity} for $\Phi$ near $x$, such that
\begin{align}
	\label{eq:loc_mod_cty_1}
	\text{for all $x' \in X$,}
	& \quad d_{Y} ( f(x'), f(x) ) \leq \omega_{f, x} \bigl( d_{X} (x', x ) \bigr) \\
	\label{eq:loc_mod_cty_2}
	\text{and}
	& \quad \omega_{f, x} (r) \to 0 \text{ as } r \to 0 .
\end{align}
In particular, \eqref{eq:loc_mod_cty_2} implies that, for each $x \in X$, there exists $r_{x} > 0$ such that $\omega_{f, x} (r_{x})$ is finite for all $0 \leq r \leq r_{x}$.
It is no loss of generality to assume that $\omega_{f, x}$ is an increasing function.
Local uniform continuity is slightly but strictly stronger than $f$ being continuous: according to \citet[Theorem~1]{Izzo1994}, on every infinite-dimensional separable normed space there exist bounded, continuous real-valued functions that are nowhere locally uniformly continuous;
however, \citet[Theorem~4]{Izzo1994} also shows that every continuous real-valued function on a metric space can be approximated uniformly by locally uniformly continuous functions.

\begin{lemma}[\ac{OM} functionals for reweighted measures]
	\label{lem:reweighted_OM_loc_unif_cts}
	Let $X$ be a metric space and let $\mu_{0} \in \prob{X}$ have an \ac{OM} functional $I_{0}$ on $E \subseteq X$.
	Let $\Phi \colon X \to \Reals$ be locally uniformly continuous.
	Suppose that $Z \defeq \int_{X} e^{-\Phi(x)} \, \mu_{0} (\rd x) \in (0, \infty)$.
	Then $\mu \in \prob{X}$ defined by $\mu (\rd x) \defeq Z^{-1} e^{-\Phi(x)} \, \mu_{0} (\rd x)$ has $I \defeq \Phi + I_{0}$ as an \ac{OM} functional on $E$.
	If, furthermore, property $M (\mu_{0}, E)$ holds and $I_{0}$ is extended to $I_{0} \colon X \to \eReals$ by setting $I_{0} \equiv + \infty$ on $X \setminus E$, then property $M (\mu, E)$ holds and $\Phi + I_{0}$ is an extended \ac{OM} functional for $\mu$.
\end{lemma}

\begin{proof}
	Let $u, v \in E$ and let $r_{0}$ be small enough that both $\omega_{\Phi, u} (r_{0})$ and $\omega_{\Phi, v} (r_{0})$ are finite;
	without loss of generality we henceforth consider only $0 < r < r_{0}$.
	By \eqref{eq:loc_mod_cty_1}, on $\cBall{u}{r}$, $\Phi$ satisfies the bound
	\[
	\Phi(u) - \omega_{\Phi, u} (r) \leq \Phi (\quark) \leq \Phi(u) + \omega_{\Phi, u} (r) .
	\]
	Since $\mu (A) = Z^{-1} \int_{A} \exp (- \Phi(x)) \, \mu_{0} (\rd x)$ for each measurable set $A \in \Borel{X}$, it follows that
	\[
	Z^{-1} \exp ( - \Phi(u) - \omega_{\Phi, u} (r) ) \mu_{0} ( \cBall{u}{r} ) \leq \mu ( \cBall{u}{r} ) \leq Z^{-1} \exp ( - \Phi(u) + \omega_{\Phi, u} (r) ) \mu_{0} ( \cBall{u}{r} ) .
	\]
	Similar arguments apply mutatis mutandis for $v$ in place of $u$.
	Hence,
	\[
	\frac{ \exp ( - \Phi(u) - \omega_{\Phi, u} (r) ) \mu_{0} ( \cBall{u}{r} ) }{ \exp ( - \Phi(v) + \omega_{\Phi, v} (r) ) \mu_{0} ( \cBall{v}{r} ) }
	\leq
	\frac{ \mu ( \cBall{u}{r} ) }{ \mu ( \cBall{v}{r} ) }
	\leq
	\frac{ \exp ( - \Phi(u) + \omega_{\Phi, u} (r) ) \mu_{0} ( \cBall{u}{r} ) }{ \exp ( - \Phi(v) - \omega_{\Phi, v} (r) ) \mu_{0} ( \cBall{v}{r} ) } ,
	\]
	and so, since $I_{0}$ is an \ac{OM} functional for $\mu_{0}$ and both $\omega_{\Phi, u} (r)$ and $\omega_{\Phi, v} (r)$ tend to $0$ as $r \to 0$,
	\[
	\exp( - \Phi(u) + \Phi(v) + I_{0}(v) - I_{0}(u) )
	\leq
	\lim_{r \searrow 0} \frac{ \mu ( \cBall{u}{r} ) }{ \mu ( \cBall{v}{r} ) }
	\leq
	\exp( - \Phi(u) + \Phi(v) + I_{0}(v) - I_{0}(u) ).
	\]
	which proves the first claim.

	The second claim is an immediate consequence of the first part and \Cref{lem:vanishing_small_ball_prob}\ref{item:vanishing_small_ball_prob_transfer}.
\end{proof}

\begin{lemma}[Continuous convergence of potentials via projection]
	\label{lem:continuous_convergence_potentials}
	Let $X$ be a separable Banach space and let $(X_{n})_{n \in \Naturals}$ be a sequence of (not necessarily nested) finite-dimensional subspaces with surjective uniformly bounded linear projection operators $P_{n} \colon X \to X_{n}$ such that
	\begin{equation}
		\label{eq:convergence_of_projections}
		\text{for all $x \in X$,}
		\quad
		\lim_{n \to \infty} \norm{ P_{n} x - x } = 0 .
	\end{equation}
	Let $\Phi \colon X \to \Reals$ be locally uniformly continuous.
	\begin{enumerate}[label=(\alph*)]
		\item \label{lem:continuous_convergence_potentials_lucts} For each $n \in \Naturals$, $\Phi \circ P_{n}$ is locally uniformly continuous.
		\item \label{lem:continuous_convergence_potentials_cvgt} $\Phi \circ P_{n} \to \Phi$ continuously as $n \to \infty$.
	\end{enumerate}
\end{lemma}

\begin{proof}
	Let $M \geq 1$ be a uniform upper bound for the operator norms $\norm{ P_{n} }$, $n \in \Naturals$.
	(Note that, in the special case that $X$ is a separable Hilbert space with complete orthonormal system $\{ \psi_{n} \}_{n \in \Naturals}$ and $P_{n}$ is the orthogonal projection onto $\spn \{ \psi_{1}, \dots, \psi_{n} \}$, we may take $M = 1$.)

	To show \ref{lem:continuous_convergence_potentials_lucts}, fix $n \in \Naturals$ and $x \in X$ and let $\omega_{\Phi, P_{n} x} \colon [0, \infty) \to [0, \infty]$ be an increasing local modulus of continuity for $\Phi$ near $P_{n} x$.
	Then, for all $x' \in X$,
	\begin{align*}
		\absval{ (\Phi \circ P_{n}) (x') - (\Phi \circ P_{n}) (x) }
		\leq \omega_{\Phi, P_{n} x} \bigl( \norm{ P_{n} x' - P_{n} x } \bigr)
		\leq \omega_{\Phi, P_{n} x} \bigl( \norm{ P_{n} } \norm{ x' - x } \bigr) .
	\end{align*}
	Thus, $\omega_{\Phi \circ P_{n}, x} (r) \defeq \omega_{\Phi, P_{n} x} ( M r )$ is a local modulus of continuity for $\Phi \circ P_{n}$ near $x$.

	To establish \ref{lem:continuous_convergence_potentials_cvgt}, fix $x \in X$ and let $\omega_{\Phi, x} \colon [0, \infty) \to [0, \infty]$ be an increasing local modulus of continuity for $\Phi$ near $x$.
	Let $\varepsilon > 0$ be arbitrary and let $r_{\varepsilon} > 0$ be such that $\omega_{\Phi, x} (r_{\varepsilon}) < \varepsilon$.
	By \eqref{eq:convergence_of_projections}, there exists $N \in \Naturals$ such that, for all $n \geq N$, $\norm{ P_{n} x - x } < r_{\varepsilon} / 2$.
	Then, for $n \geq N$ and $x' \in X$ with $\norm{ x' - x } < r_{\varepsilon} / 2 M$,
	\begin{align*}
		\absval{ (\Phi \circ P_{n}) (x') - \Phi(x) }
		& \leq \omega_{\Phi, x} \bigl( \norm{ P_{n} x' - x } \bigr) \\
		& \leq \omega_{\Phi, x} \bigl( \norm{ P_{n} x' - P_{n} x } + \norm{ P_{n} x - x } \bigr) \\
		& \leq \omega_{\Phi, x} \bigl( M \norm{ x' - x } + \norm{ P_{n} x - x } \bigr) \\
		& \leq \omega_{\Phi, x} (r_{\varepsilon}) < \varepsilon ,
	\end{align*}
	as required.
\end{proof}

\section*{Acknowledgements}
\addcontentsline{toc}{section}{Acknowledgements}

BA and TJS are supported in part by the \ac{DFG} through project \href{https://gepris.dfg.de/gepris/projekt/415980428}{415980428}.
Portions of this work were completed during the employment of BA and TJS at the Freie Universit\"at Berlin and while guests of the Zuse Institute Berlin, and during the employment of IK at the Zuse Institute Berlin.
IK and TJS have been supported in part by the \ac{DFG} through projects TrU-2 and EF1-10 of the Berlin Mathematics Research Centre MATH+ (EXC-2046/1, project \href{https://gepris.dfg.de/gepris/projekt/390685689}{390685689}).
\change{The research of HCL has been partially funded by the \ac{DFG} --- Project-ID \href{https://gepris.dfg.de/gepris/projekt/318763901}{318763901} --- SFB1294.}
\change{The authors thank two anonymous peer reviewers for their helpful suggestions.}

\bibliographystyle{abbrvnat}
\bibliography{references}
\addcontentsline{toc}{section}{References}

\end{document}